\documentclass[12pt,twoside]{amsart}
\usepackage{amsmath}
\usepackage{amssymb}
\usepackage{graphicx,epsf,amsmath}  
\usepackage{epsf,graphicx}
\usepackage{comment}

\newtheorem{thm}{Theorem}[section]
\newtheorem{theorem}{Theorem}[section]
\newtheorem{pro}[thm]{Proposition}
\newtheorem{conj}[thm]{Conjecture}
\newtheorem{lem}[thm]{Lemma}
\newtheorem{lemma}[thm]{Lemma}

\newtheorem{corollary}[thm]{Corollary}

\newtheorem{defi}[thm]{Definition}

\newtheorem{preremark}[theorem]{Remark}
\newenvironment{remark}{\begin{preremark}\rm}{\end{preremark}}
\def\NN{{\mathbb N}}

\def\ZZ{{\mathbb Z}}
\def\integer{{\mathbb Z}}
\def\RR{{\mathbb R}}
\def\real{{\mathbb R}}
\def\TT{{\mathbb T}}
\def\torus{{\mathbb T}}
\def\CC{{\mathbb C}}
\def\complex{{\mathbb C}}
\def\NN{{\mathbb N}}

\def\ZZ{{\mathbb Z}}
\def\integer{{\mathbb Z}}
\def\RR{{\mathbb R}}
\def\real{{\mathbb R}}
\def\TT{{\mathbb T}}
\def\torus{{\mathbb T}}

\def\CC{{\mathbb C}}
\def\complex{{\mathbb C}}
\def\M{{\mathcal M}}
\def\Id{{\rm Id}}
\def\ep{\varepsilon}
\def\A{{\mathcal A}}
\def\D{{\mathcal D}}
\def\H{{\mathcal H}}
\def\W{{\mathcal W}}
\def\th{\theta} 

\def\NN{{\mathbb N}}

\def\ZZ{{\mathbb Z}}
\def\integer{{\mathbb Z}}
\def\RR{{\mathbb R}}
\def\real{{\mathbb R}}
\def\TT{{\mathbb T}}
\def\torus{{\mathbb T}}

\def\CC{{\mathbb C}}
\def\complex{{\mathbb C}}

\def\L{{\mathcal L}}
\def\M{{\mathcal M}}
\def\N{{\mathcal N}}
\def\S{{\mathcal S}}
\def\U{{\mathcal U}}
\def\Id{{\rm Id}}

\def\avg{{\rm avg\,}}

\def\ep{\varepsilon}
\def\A{{\mathcal A}}
\def\X{{\mathcal X}}
\def\th{\theta}

\def\dist{{\rm dist}}

\def\Tau{{\mathcal T}}

\title[Whiskered tori in PDEs]{An {\sl a posteriori} KAM theorem for
  whiskered tori in Hamiltonian partial differential equations with applications to some ill-posed  equations}
\author[R. de la Llave]{Rafael de la Llave}
\address{School of Mathematics, Georgia Institute of 
Technology, Atlanta GA 30332}
\email{rafael.delallave@math.gatech.edu}
\thanks{R.L. Supported in part by National Science Foundation DMS-1500943}
\author[Y. Sire]{Yannick Sire} 
\address{Johns Hopkins University, Krieger Hall, Baltimore, USA}
\email{sire@math.jhu.edu}

\begin{document}
\begin{abstract}
The goal of this paper
is to develop a KAM theory for tori with hyperbolic directions, 
which applies 
to Hamiltonian partial differential equations, even to 
some  ill-posed ones. 

The main result has an \emph{a-posteriori} format, 
i.e., we show that if there is an approximate solution of 
an invariance equation which also satisfies some 
non-degeneracy conditions, then there is a true solution nearby. 
This allows, besides dealing with 
the quasi-integrable case, 
 to validate numerical computations or formal perturbative
expansions as well as to obtain quasi-periodic  solutions in  
degenerate situations. The a-posteriori format 
also has other automatic consequences
(smooth dependence on parameters, bootstrap of
regularity, etc.).  We emphasize that the non-degeneracy conditions 
required are just quantities evaluated on the approximate solution
(no global assumptions on the system such as twist). Hence, 
they are readily verifiable in perturbation expansions. 

The method of proof is based on an iterative method to solve a
functional equation for the parameterization of the torus satisfying 
the invariance equations  and for parametrization of 
directions 
invariant under the linearizatation. The iterative method does
 not use transformation
theory or action-angle variables.
It does  not assume that the system is close to integrable. We do not even need
that the equation under consideration admits solutions for every
initial data. In this paper we present in detail 
the case of analytic tori when the equations are analytic 
in a very weak sense. 

We first develop an abstract theorem. Then, we show 
how this abstract result applies to some concrete 
examples, including the scalar Boussinesq equation 
and  the Boussinesq system so that 
we  construct {\sl small amplitude} tori for the
equations, which are even in the spatial variable.  Note that the
 equations we use as examples are
ill-posed.  The strategy for 
the abstract theorem is inspired by that 
 in \cite{fontichdelLS07a,fontichdelLS071}. The main part of the paper
is to study infinite dimensional analogues of dichotomies 
which applies even  to ill-posed equations and which is 
stable under addition of unbounded perturbations. This requires
that we assume smoothing properties.  We also present 
very detailed bounds on the change of the splittings  under 
perturbations.

\end{abstract}
\maketitle
\tableofcontents
\section{Introduction}
The goal of this paper
is to develop a KAM theory for tori with hyperbolic directions, 
which applies 
to Hamiltonian partial differential equations, even to 
some  ill-posed ones. The main result, Theorem~\ref{existence} is 
stated in an a-posteriori format, that is, we formulate 
invariance equations and show that approximate solutions 
that satisfy some explicit non-degeneracy conditions, 
lead to a true solution. This a-posteriori format 
leads automatically to several consequences (see Section~\ref{sec:consequences}) 
and can be used to justify numerical solutions 
and asymptotic expansions. We note that the results do not 
assume that the equations we consider define evolutions and indeed
we present examples of quasi-periodic solutions in 
some well known ill-posed equations. See Sections~\ref{applications}, 
\ref{applications2}.

\subsection{Some general considerations and relations 
with the literature} 
Some partial differential equations appear as models 
of evolution in time for Physical systems. 
It is 
natural to consider such evolutionary PDE's as a dynamical 
system and  try to use the methods of dynamical 
systems. 

Adapting dynamical systems techniques to evolutionary 
PDE's has to overcome several technical difficulties.
For starters, since the PDE's involve 
unbounded operators, the standard theory of existence, uniqueness developed 
for ordinary differential equations does not apply. As it is well 
known, by now, 
there are systematic ways of defining the evolution
using e.g. semigroup theory \cite{Showalter, pazy,Goldstein85} and 
many dynamical systems techniques can be adapted in 
the generality of semigroups ( see the pioneering work of  \cite{henry}
and more modern treatises
 \cite{Hale88, Miyadera92,Temam97,ConstantinFNT89, Robinson01, SellY02,
ChepyzhovV02, HaleMO02,CherrierM12}.)
Besides the analytic difficulties, adapting ODE techniques 
to PDE's has to  face that several geometric 
arguments fail to hold. 
 For instance, symplectic structures on
infinite-dimensional spaces (see for instance
\cite{chernoffM74,bambusi99}) could  lack several important properties.
Hence, the techniques (e.g. KAM theory) that are based on geometric 
properties have to overcome several difficulties
 specially the methods based on transformation theory
\cite{Kuksin93, Kuksin94, Kuksin00,  Kuksin06,KappelerP03}. 
Some recent methods based on avoiding transformation theory 
are \cite{CraigW1,CraigW2,Bourgain00,Berti07,Craig00}. When working near an equilibrium 
point, one also has to face the difficulty that the action angle variables 
are singular (even in finite dimensions) 
\cite{KappelerP03, GrebertK14}. In the approach of this paper, we
do not use action angle variables, which present difficulties even in 
finite dimensional fixed points and, much more in PDE's. 

One class of evolutionary equations that has not received much
systematic attention is \emph{ill-posed} equations. In ill-posed
equations, one cannot define the evolution for all the initial data in
a certain space (an equation may be ill-posed in a space and well
posed in another) or the evolution is not continuous in this space.
Nevertheless, it can be argued that even if one cannot find solutions
for all the inital data, one can still find interesting solutions
which provide accurate descriptions of physical phenomena.  Many
ill-posed equations in the literature are obtained as a heuristic
approximation of a more fundamental equation.  The solutions of the
ill-posed equation may be approximate solutions of the true equation.

For example, many \emph{long wave} approximations of water waves 
turn out to be ill-posed (e.g. the Boussinesq equations used as
examples here, see Section~\ref{applications}) but several 
special solutions (e.g. traveling waves or the quasi-periodic solutions 
considered in this paper) of the long wave approximations  can be constructed. 
These special solutions are such that, for them, 
the long wave approximation is
rather accurate. Hence, 
the solutions obtained  here for  the long wave 
approximation provide approximate solutions of the original  water wave 
equation and are physically relevant.

Note that the long-wave approximations are PDE's while 
the water waves problem is a free boundary and many techniques are
different, notably in numerical analysis. Being able to validate the 
numerical solutions is useful. 

Of course, the straightforward adaptation  of
ODE methods for invariant manifolds
to ill-posed equations present some challenges because 
some methods (e.g. graph transform, index theory methods, etc.), which are very useful in ODEs,  require 
taking  arbitrary initial conditions.  Nevertheless, we will present 
rather satisfactory adaptations of some of the methods of hyperbolic 
dynamical systems.

In the present paper, we are concerned with the construction of quasi-periodic
motions of PDEs. The method is very general.
Some concrete examples of 
ill-posed equations to which the method applies 
will be presented in Sections~\ref{applications} and \ref{applications2}. 

The tori we consider are \emph{whiskered}, that is the linearization has 
many hyperbolic directions, indeed, as many directions as it is 
possible to be compatible with the preservation of the symplectic 
structure. There is a rich KAM theory for whiskered tori 
\cite{Graff74,Zehnder75b} or for lower dimensional tori will 
elliptic directions \cite{Eliasson88,You99, LiY05,Sevryuk06}. 
A treatment of normally elliptic tori by methods similar 
to those here is in \cite{LuqueV11}. 

In PDE's, where the phase space is infinite dimensional, 
the quasiperiodic solutions are very low dimensional. Nevertheless, 
most of the literature in PDE is concerned with normally elliptic tori,
so that most of the  small divisors come from the elliptic 
 normal directions. The models
considered here have no elliptic normal directions. 
 On the other hand, the models
we consider do not admit solutions for all initial conditions and present
very severe unstable terms.  Hence, methods based on transformation theory, 
normal forms etc. are very difficult in our case.  We also deal 
with unbounded perturbations.

\subsection{Overview of the method}
We are going to follow roughly the method
described in \cite{fontichdelLS07a} 
and implemented in \cite{fontichdelLS071} for finite dimensional 
systems,  in \cite{LiL09,FontichLS15} for infinite dimensional 
systems (but whose evolution is 
a smooth differential equation; the main 
difficuly overcome in \cite{LiL09} was 
the fact that the equations involve delays, 
a new difficulty in \cite{FontichLS15} is the spatial 
structure). In this paper we overcome the difficulty 
that the evolution equations are PDE's which are perturbed by unbounded 
operators. Hence, we  have to overcome many  problems
(unbounded operators, regularity issues and spectral theory for instance ). Some results in KAM with unbounded 
perturbations by very different methods appear in \cite{LiuY11}.

The method we use is based on 
the solution of a functional equation whose unknown is a
parameterization of the invariant torus
and devicing a Newton method to solve these equations by quadratically 
convergent schemes. We assume that the linearized evolution admits 
an invariant splitting. In the hyperbolic directions we can use
essentially soft functional analysis methods. There are subtleties 
such that we have to deal with unbounded perturbations and be very 
quantitative in the hyperbolic perturbation theory, and a center 
direction case, in which we have to deal with equations involving 
small divisors and use heavily the number theoretic properties of 
the equation and the symplectic geometry. 

The method does not rely on methods that require the evolution for all
initial data on a ball. 
Also, the symplectic geometry properties are 
used only  sparingly. We certainly  do no use action-angle variables. 
The solutions we construct are very unstable -- indeed, some perturbations 
near them may lead to a solution of the evolution 
equation -- but they are in some precise sense 
hyperbolic in the usual meaning of dynamical systems. 
 We expect that one can define stable and unstable 
manifolds for them and we hope to come back to this problem. 
Fortunately, the analysis on the center is very similar to the 
analysis in the finite dimensional case. The bulk of the work is 
in the study of hyperbolic splittings with unbounded perturbations. 
We hope that the theory developed here can be used in other contexts. 

Indeed, other theories of persistence of
invariant splitting (having significant applications
to PDE) have  already been developed 
in \cite{henry, PlissS99, ChowL95, ChowL96,HaragusI11}. 
The main difference between Section~\ref{change-nondeg} and 
 \cite{ChowL95, ChowL96} is that we take advantage of the smoothing properties 
and, hence, can deal with more singular perturbations. We also 
take advantage of the fact that the dynamics on the base is a rotation
whereas \cite{ChowL95,ChowL96} deal with more general dynamics. 
This allows us to obtain analyticity results which are false in 
the more general contexts considered in \cite{ChowL95, ChowL96}.

The method presented here applies even to some ill-posed equations. A
fortiori, it applies also to well posed equations.  Even then, it
presents advantages, notably our main result has an a-posteriori
format that can justify several expansions and deal with situations
with weak hyperbolicity, bootstrap regularity, establish smooth
dependence, etc. It also leads to efficient numerical algorithms. See
Section~\ref{sec:consequences}. In a complementary direction, we point
out that for finite dimensional problems the present methods leads to
efficient algorithms (See \cite{HuguetLS12}). The case without center
directions and no Hamiltonian structure has been considered in
\cite{CanadellH15}.

\subsection{Organization of the paper}

This paper is organized as follows: In Section~\ref{sec:overview} we
present an overview of the method, describing the steps we will take,
but ignoring some important precisions (e.g. domains of the
operators), and proofs. In Section~\ref{sec:framework} we start
developing the precise formulation of the results.  We first present
an abstract framework in the generality of equations defined in Banach
spaces, including the abstract hypothesis.  The general abstract
results are stated in Section~\ref{sec:statement-general} and in
Section~\ref{sec:results-particular} we discuss how to apply the
results to some concrete examples.  Some possible extensions are
discussed in Section~\ref{sec:consequences}.  The rest of the paper is
devoted to the proof of the results following the strategy mentioned
in the previous sections. One of the main technical results, which
could have other applications is the persistence of hyperbolic
evolutions with smoothing properties.  See
Section~\ref{change-nondeg}.

\section{Overview of the method} 
\label{sec:overview}

In this section, we present a quick overview describing 
informally  the steps 
of the method.  We present the equations that need
to be solved and the manipulations that need to be done ignoring issues 
such as domain of operators, estimates. These precisions  will be taken up
in Section~\ref{sec:framework}. This section can serve as motivation for
Section~\ref{sec:framework} since we use the 
formal manipulations to identify the issues that need to 
be resolved by a precise formulation.

We will discuss first abstract results, but in 
Sections~\ref{applications} and \ref{applications2}, we will show that the abstract 
result applies to concrete examples.

One example to keep in mind 
and which has served as an important motivation for us  is 
the Boussinesq equation 
\begin{equation}\label{bou0}
u_{tt}= \mu u_{xxxx}+u_{xx}+(u^2)_{xx}\,\,\,\, x \in  \torus,\,\,t\in \RR, \mu >0
\end{equation} 
In Section~\ref{applications2}, we will also consider the 
Boussinesq system. Other models in the literature which fit our scheme are 
the Complex Ginzburg-Landau equation and the derivative 
Complex Ginzburg-Landau equation for values of the parameters in 
suitable ranges.

\begin{remark} 
There are several equations called the
Boussinesq equation in the literature 
(in Section \ref{applications2} we also present the
Boussinesq system), notably the Boussinesq equation for fluids under
thermal buoyancy.  The paper \cite{McKean} uses the name Boussinesq
equation for $u_{tt}= - u_{xxxx}+(u^2)_{xx}$ and shows it
is integrable in some sense made precise in that paper. Note that this equation is
very different from \eqref{bou0} because of the sign of the fourth
space derivative and (less importantly), the absence of the term with
the second derivative.  The sign of the fourth derivative term causes
that the wave propagation properties of \eqref{bou0} and the equation
in \cite{McKean} are completely different. 

Sometimes people refer to
\eqref{bou0} with $\mu >0$ as the \emph{``bad''} Boussinesq equation,
and call the equation with $\mu < 0$, the \emph{``good''} Boussinesq 
equations. We note that the case $\mu > 0$ considered here
 is the case that appears in water waves 
(see \cite[Equation  (26)]{Boussinesq} ).  
\end{remark} 

\begin{remark}
We note that the fourth derivative in \eqref{bou0} is just the next term in 
the long wave expansion of the water wave problem 
(which is not a PDE, but rather a free boundary 
problem).  Equations similar to \eqref{bou0} appear 
in many long wave approximations for waves. 
See \cite{CraigGNS05,Craig08} for  modern discussions. 

The special solutions of \eqref{bou0} which are in the range 
of validity of the long wave approximation are  good approximate solutions of 
the water wave problem, but they are analyzable by PDE methods rather 
than the free boundary methods required by  the original problem. 
\cite{ChenNS11,LewickaM09}. Note that the solutions produced here lie
in the regime (low amplitude, long wave) where the equation \eqref{bou0} 
was derived, so that they provide approximate solutions  to 
the water wave problem. 
\end{remark}

\subsection{The evolution equation}
We consider an evolutionary 
PDE, which we write symbolically, 
\begin{equation} 
\label{eq:evolutionary-formal} 
\frac{d u}{dt}  = \X\circ u 
\end{equation} 
where $\X$ will be a differential and possibly non-linear  operator. 
This will, of course, require assumptions on domains etc. which 
we will take up in Section~\ref{sec:framework}.  For the moment, we will just 
say that $\X$ is defined in a domain inside a Banach space $X$. 
We will write
\begin{equation} \label{decomposition} 
\X(u) = \A u  + \N(u) 
\end{equation} 
where $\A$ is linear and $\N$ is a  nonlinear and  possibly 
unbounded operator.

The differential equations $\dot u =\A u$ 
will  not be assumed 
to generate dynamical evolution for all initial conditions
(we just assume that it  generates forward
and backward evolutions when restricted to appropriate subspaces). 
Of course, we will not assume that \eqref{eq:evolutionary-formal}
defines an evolution either. Lack of 
solutions for all the initial conditions will not be a severe problem for 
us since we will only try to produce some specific solutions.

The meaning in which \eqref{eq:evolutionary-formal} is to 
hold may  be  taken to be  the classical sense. As we will see we will 
take the space $X$ to consist of very differentiable functions
so that the derivatives can be taken in the elementary classical 
sense. As intermediate steps, we will also find useful some 
solutions in the \emph{mild} sense, satisfying some 
integral equations formally equivalent to \eqref{eq:evolutionary-formal}. 
The mild solutions require less regularity in $X$. 
Again, we emphasize that the solutions we try to produce 
are only special solutions. 

We will assume that the nonlinear operator $\N$ is \emph{``sub-dominant''}
with respect to the linear part. This will be formulated later
in Section~\ref{sec:framework}, but 
we anticipate that this means  roughly that
$\A$ is of higher order than $\N$ and 
that the evolution generated by  $\A$ when restricted to appropriate 
sub-spaces gains more derivatives than  the order of $\N$.  We will 
formulate all this precisely later.

We will follow \cite{henry} and formulate these effects 
by saying that the operator $\mathcal N$ is an analytic function 
from a domain $\U \subset X$ -- $X$ is a Banach space of smooth functions -- to $Y$  -- a space corresponding to 
less smooth functions and that the evolution operators map $Y$ back 
to $X$ with some quantitative bounds. 

In the applications that we present in 
Sections~\ref{applications} and \ref{applications2}, the equations we consider are
polynomial\footnote{The equations we consider are taken from 
the literature of approximations of water waves. In these derivations, 
it is customary to expand the non-linearity and keep only the lower
order terms} but the method can deal with more general nonlinearities.

\subsection{The linearized evolution equations}
Note that, in this set up we can define a linearized evolution 
equation around a curve $u(t)$ in $X$, i.e. 
\begin{equation}
\label{linearized}
\frac{d \xi}{dt}  = D\X \circ u(t)  \xi \equiv \A \xi + D\N(u(t)) \xi
\end{equation}
The equations \eqref{linearized} are to be considered as evolution equations for $\xi$ while 
$u(t)$ is given and fixed.  The meaning of the term $D\N$ could be understood if 
$\N$ is a differentiable operator from $X$ to $Y$. 

Of course, when $u(t)$ is solution of the evolution equation  
\eqref{eq:evolutionary-formal}, equations
\eqref{linearized} are the variational equations for the evolution. 
In our case, the evolution is not assumed to exist and, 
much less,  the variational equations are assumed to provide a description of 
the effect of the initial conditions on the variation.  We 
use these equations  \eqref{linearized}
 even  when $u(t)$ is not a 
solution of the evolution equation 
\eqref{eq:evolutionary-formal} and we will show that they are 
indeed a tool to modify an approximate solution $u(t)$ 
into a true solution.

Notice that \eqref{linearized} is non-autonomous, linear 
non-homogeneous, but that the existence of solutions is not guaranteed 
for all the initial conditions (even if the time dependent term is 
omitted).

In the finite dimensional case, equations of the form
\eqref{linearized} even when $u(t)$ is not a 
solution  are studied when performing a Newton method to construct a
solution; for example in 
multiple shooting. Here, we will use \eqref{linearized} in a similar way. We will see that \eqref{linearized}
can be studied using that $\A$ is dominant and has a splitting (and
that $u(t)$ is not too wild).

\subsection{The invariance equation}

Given a  fixed $\omega \in \real^\ell$ that satisfies some good
number theoretic properties (formulated precisely
 in Section~\ref{sec:Diophantine}), 
we will be seeking an embedding $K: \torus^\ell \rightarrow X$ in such 
a way that 
\begin{equation}
\label{eq:embedding} 
\X \circ K = DK \cdot \omega
\end{equation}
Note that if \eqref{eq:embedding}  holds, then,  for
any $\theta \in \torus^\ell$, 
$u(t) = K(\omega t + \theta)$ will be  solution of 
\eqref{eq:evolutionary-formal}.  
Hence, when we succeed in producing a solution of 
\eqref{eq:embedding}, we will have a $\ell$-parameter family of quasi-periodic 
solutions. The meaning of these parameters is the origin of the  phase as 
is very standard in the theory of quasi-periodic functions.

\subsection{Outline of the main result}

The main ingredient of the main result, Theorem~\ref{existence}  is that we will assume given an 
approximate solution $K_0$ of \eqref{eq:embedding}. 
That is, we are given an embedding $K_0$ in such a way that 
\begin{equation}
\label{eq:embedding-approx} 
\X \circ K_0 -  DK_0 \cdot \omega \equiv e 
\end{equation}
is small enough. We will also assume that the 
linearized evolution 
 satisfies some non-degeneracy assumptions. The conclusions is 
that  there is a true solution close to the original approximate solutions. 
Theorems of these form in which we start from an approximate solution
and conclude the existence of a true one are often called \emph{`` a posteriori''} theorems. 

In the concrete equations that we consider in the applications, the approximate solutions will be 
constucted using Lindstedt series. 

The sense in which the error $e$ is small requires defining
appropriate norms, which will be taken up in 
Section~\ref{sec:framework}. The precise form 
of the non-degeneracy conditions will be 
motivated by the following discussion which specifies the steps we 
will perform for the Newton method for the linearized equation

\begin{equation}
\label{eq:invariance-linearized} 
\frac{d u}{dt}  = D\X \circ K_0(\theta + \omega t) u 
\end{equation}

The non-degeneracy conditions have two parts. 
We first  assume that for each $\theta \in \torus^\ell$, 
 the linearized equation 
satisfies some spectral properties. 
These spectral properties mean roughly  that there  are solutions of 
\eqref{eq:invariance-linearized}  that decrease exponentially in 
the future (stable solutions), others that decrease 
exponentially in the past (unstable solutions),   and some center directions
that can grow or decrease with a smaller exponential rate. 
The span of these three class of solutions is the whole space. 
We will also assume that the evolutions, when they can be defined, 
gain regularity. 

In the ODE  case, this means that the linearized 
equation admits an exponential trichotomy in the sense of 
\cite{SackerS76b}. 

In the PDE case, there are some subtleties not present in the ODE
case.  For instance, the vector field is not differentiable and is
only defined on a dense subset.

We will 
not assume that \eqref{eq:invariance-linearized} defines 
an evolution for all time and all the initial conditions. 
We will however assume that \eqref{eq:invariance-linearized}
admits a solution forward in time for initial conditions in 
a space (the center stable space) and backwards in time 
for the another space (the center unstable space). 
We will furthermore assume that the center stable and center unstable
 spaces span 
the whole space, and they have a finite dimensional intersection
(we will also assume that they have a finite angle, which we 
will formulate as saying that the projections are bounded). 
We emphasize that we will not assume that the evolution forward of 
\eqref{eq:invariance-linearized} can be defined outside of the 
center stable space nor that the backward evolution  can
be defined outside of the center unstable space.

Furthermore, we will assume that the 
evolutions defined in these spaces are smoothing. Of course, these subtleties 
are only present when we consider evolutions generated by unbounded 
operators and are not present in the ODE case.

A crucial  result for us  is Lemma~\ref{iterNH} which shows 
that this structure (the trichotomy 
with smoothing)  is stable  under the addition of unbounded 
terms of lower order. 
We also present very quantitative estimates on 
the change of the structure under perturbations. Note that 
the result is also presented in an a-posteriori format so 
that we can use  just the existence of an approximate 
invariant splitting. 

The smoothing properties along the 
stable directions  overcome the loss of regularity of 
the perturbation. Hence, we can obtain a persistence of 
the spaces under unbounded perturbations of 
lower order.  A further argument shows the persistence of the
smoothing properties. 
The result in 
Lemma~\ref{iterNH}  can be considered as a generalization of the 
finite dimensional result on stability of 
exponential dichotomies to allowing unbounded 
perturbations. 
 An important consequence is that, when $\N(u)$ is small enough
(in an appropriate sense)  we
can transfer the hyperbolicity from $\A$ to the approximate solution, 
which is the way that we construct the approximately 
hyperbolic solutions in the applications.

We will need to assume that in the center directions, there is some
geometric structure that leads to some cancellations (sometimes called
\emph{automatic reducibility}). These cancellations happen because of
the symplectic structure. We note that, in our case, we only need a
very weak form of symplectic structure, namely that it can be made
sense of in a finite dimensional space consisting of rather smooth
functions.  Note that the infinitesimal perturbations do not grow in
the tangent directions. The preservation of the geometric structure
also implies that some of the perpendicular directions evolve not
faster than linearly. Hence, the tori we consider are never normally
hyperbolic and that for $\ell$-dimensional tori, the space of
directions with subexponential growth is at least $2\ell$
dimensional. We will assume that the tori are as hyperbolic as
possible while preserving of the symplectic structure. That is, the
set of directions with subexponential growth is precisely $2 \ell$
dimensional . These tori are called \emph{whiskered} in the finite
dimensional case.

We note that the geometric structure we need only requires 
to make sense as the restriction to an infinitesimal space 
and be preserved only in a set of directions. The geometric structure 
that appears naturaly in applications will be given by 
an unbounded form and many of the deeper features of 
symplectic structures in finite dimensions will not be available. 
Hence, it is important to note that the present method does 
not rely much in the symplectic structure. We do not rely on transformation 
theory we only use some geometric identities in finite dimensional spaces
to construct a good system of coordinates in finite dimensions
and to show that some (finite dimensional) averages vanish.  
In systems without the geometric structure, the system of coordinates
and the averages would require adjusting external parameters.

\begin{remark} 
 We  note that \eqref{eq:invariance-linearized} 
is formally the variation equation giving the derivative
of the flow of the evolution equation. 
This interpretation is very problematic since the equations we 
will be interested in do not define necesserally a flow. 

 An important part 
 of the effort in 
Section~\ref{sec:framework} consists in  defining these 
structures in the restricted framework considered in this paper
when many of 
the geometric operations  used in the finite dimensional case
are not available.
\end{remark}

We also need to make assumptions that are analogues of 
the twist conditions in finite dimensions. See Definition~\ref{ND2}. 
 The twist condition we 
will require is just that 
a  finite dimensional matrix is invertible. The matrix
is computed explicitly on the approximate solution and does
not require any global considerations on the differential equation.

\subsection{Overview of the proof} 
\label{sec:overview-proof}

The method of proof will be to show that, under the hypotheses we are making,
a quasi-Newton method for equation \eqref{eq:embedding} 
started in the initial guess, 
converges to a true solution. We emphasize that the unknown in 
equation  \eqref{eq:embedding}  is the embedding $K$ of $\torus^\ell$
into a Banach space $X$. 
Hence, we will need to introduce families of Banach spaces of embeddings
(the proof of the convergence 
will be patterned after  the corresponding proofs \cite{Moser66a, Zehnder75}).

For simplicity, we will only consider analytic spaces of embeddings. 
Note that the regularity of the embedding $K$ as a function of
their argument 
$\theta \in \torus^\ell$  is  different 
from  the regularity of  the functions $K(\theta) \in X$. 
The term $K(\theta)$ will be functions of 
the $x$ variable. The space $X$ encodes the regularity 
with respect to the variable $x$. Indeed, we will consider 
also other Banach spaces $Y$ consisting of functions of smaller 
regularity  in $x$. 

The Newton method consists in solving the equation 
\begin{equation}
\label{eq:Newton} 
\frac{d}{dt} \Delta(\theta + \omega t) 
-  D\X \circ K_0(\theta + \omega t) \Delta(\theta + \omega t)  = -e
\end{equation}
and then,  taking $K_0 + \Delta$ as an improved  solution. 

Clearly, \eqref{eq:Newton} is a non-homogeneous version of
\eqref{eq:invariance-linearized}. Hence, the spectral properties of
\eqref{eq:invariance-linearized} will play an important role in the 
solution of   \eqref{eq:Newton} by the variations of constants formula. 
Following \cite{fontichdelLS07a,fontichdelLS071}, 
we will show that using the trichotomy, 
we can decompose 
\eqref{eq:Newton} into three equations, each one of them corresponding to 
one of the invariant subspaces.

The equations along the stable and 
unstable directions can be readily solved using the variation of parameters
formula also known as Duhamel formula
(which holds in the generality of semigroups)
since the exponential contraction and the smoothing allow us
to represent the solution as a convergent integral. 

The equations along the center direction, as usual, are much more delicate. 
We will be able to show the geometric properties to 
establish the \emph{automatic reducibility}. That is, we will show 
that there is an explicit 
change of variables that reduces the equation along the 
center direction to the standard cohomology equations over rotations
(up to an error which is quadratic -- in the Nash-Moser sense -- 
 in the original  error in the invariance equation). 
It is standard that we can solve these cohomology equations under 
Diophantine assumptions on the rotation and that we can obtain 
\emph{tame} estimates in the standard meaning of KAM theory
\cite{Moser66a, Moser66b, Zehnder75}.  
One geometrically delicate point is that the cohomology equations admit
solutions provided that certain averages vanish. The vanishing of these 
averages over perturbations is related to the exactness properties of 
the flow. Even if this is, in principle,  much more delicate in the infinite 
dimensional case, it will turn out to be very similar to the finite 
dimensional case, because we will work on the restriction to the center 
directions which are finite dimensional.  The procedure is 
very similar to that in \cite{fontichdelLS07a}. 

We will not solve the linearized  equations in center direction 
exactly. We will solve them up to 
an error which is quadratic in the 
original error. The resulting modified Newton method, 
will still lead to  quadratically 
small error  in the sense of Nash-Moser theory
and can be used as the basis of a quadratically convergent method.

Once we have the Newton-like  step under control 
we need to show that the step can be iterated 
infnitely often and it converges to the solution of the problem. 

A necessary step in the strategy is to show  
stability of the non-degeneracy assumptions.  The stability of 
the twist conditions is not difficult since it amounts to the invertibility 
of a finite dimensional matrix, depending on the solution. 
The stability of spectral theory is reminiscent of the standard 
stability theory for trichotomies \cite{SackerS76b, HirschPS77}
but it requires significant more work since we need to use the smoothing 
properties of the evolution semigroups
to control the fact that the perturbations are unbounded. Then, we need
to recover the smoothing properties to be able to 
solve the cohomology equations. For this functional 
analysis set up,  we have found very inspiring the
 \emph{``two spaces approach''} of \cite{henry} and some of the geometric
constructions of \cite{henry,PlissS99, ChowL95,ChowL96}.
Since the present method is part of an iterative procedure,
we will need very detailed estimates of the change.

We note that rather than presenting the main result as 
a persistence result, we prove an a-posteriori result showing 
that an approximate invariant structure implies the existence of 
a truly invariant one and we bound the distance between the original
approximation and the truly invariant one. This, of course, implies
immediately the persistence results.

\section{The precise framework for the results} 
\label{sec:framework} 

In this section we formalize the framework for our  abstract results.  
 As indicated above, we will present carefully 
the technical assumptions on 
domains, etc. of the  operators under consideration, 
and the symplectic forms.  We will formulate 
spectral non-degeneracy conditions and the twist non-degeneracy assumption.

In Section~\ref{sec:statement} we will 
state our main abstract  result, Theorem~\ref{existence}. 
The proof will be obtained in the subsequent Sections. 
Then, in Sections~\ref{applications} and \ref{applications2} we will show 
how the abstract theorem applies to several examples.
The abstract framework has been chosen so that 
the examples fit into it, so that the reader is encouraged 
to refer to these sections for motivation.   Of course, the abstract 
framework has been formulated with the goal that it applies to
other problems in a more or less direct manner. We leave 
these to the reader.

We note that the formalism we use is inspired by the 
\emph{two-space formalism} of
\cite{henry}. We consider two Hilbert spaces $X$ and $Y$. The differential
operators, which are unbounded from a space to itself will be 
very regular operators considered as operators from $X$ to $Y$.
Some evolutions will have smoothing properties and map $Y$ to $X$ 
with good bounds.

\subsection{The evolution equation} 

We will consider an evolution equation 
as in \eqref{eq:evolutionary-formal} and \eqref{decomposition}. 

We assume

{\bf H1} There are two complex  Hilbert spaces

\begin{equation*}
X \hookrightarrow Y ,
\end{equation*}   
with continuous embedding.
The space $X$ (resp. $Y$) is endowed with the norm $\|.\|_X$ (resp. $\|.\|_Y$) 

We denote by 
$\L (X_1, X_2)$ 
the space of bounded linear operators from $X_1$ to $X_2$.

We will assume furthermore that $X$ is dense in $Y$. 
We will assume in applications that $\A$ and $\N$ are such that 
they map real functions into real functions; it will be 
part of the conclusions that the solutions of the invariance equations we obtain are then real.

\medskip
{\bf H2} The  non-linear part $\N$ of  \eqref{decomposition} is 
an analytic function from  
$X$ to $Y$. 

\medskip
We recall that the definition of an analytic function is 
that it is locally defined by a norm 
convergent sum 
of multilinear operators. Since we will be considering 
an implicit function theorem, it suffices to consider just one
small neighborhood and a single expansion in multi-linear operators. 
The examples in Sections~\ref{applications} and \ref{applications2} have nonlinearities 
which are just polynomials (finite sums of multilinear operators).

\begin{remark} 
In  our case, it seems that some weaker assumptions would work. 
It would suffice that $\X \circ K(\theta)$ is analytic for any 
analytic  embedding $K$. In many situations this is equivalent to 
the stronger definition \cite[Chapter III]{hilleP}. 
In the main examples that we will consider
 and in other applications, 
the vector field $\X$ is a polynomial. 
\end{remark}

\begin{remark}
It also seems possible that one could deal with finite differentiable 
problems. For the experts, we note that there are two types of KAM smoothing techniques: 
either smoothing only the solutions in the iterative processs 
(single smoothing)\cite{Schwartz, CallejaL}
or smoothing also the problems (double smoothing)
\cite{Moser66a, Zehnder75}.  In general, double smoothing 
techniques produce better differentiability in the results. On 
the other hand, in this case, the approximation of the 
problems seems fraught with difficulties (how to define smoothings 
in infinite dimensional spaces, also for unbounded operators). 
Nevertheless, single smoothing methods do not seem to have 
any problem. Of course, if the non-linearities have some special 
structure (e.g. they are obtained by composing with a non-linear function) 
it seems that a double smoothing could also be applied. 
\end{remark}

\begin{remark} 
Note that the structure of $\X$ assumed in \eqref{decomposition} allows
us to estimate always the errors in $Y$, even if the unknown $K$ are in $X$. 

This is somewhat surprising since  the loss of derivatives from 
$X$ to $Y$ is that of the subdominant term $\N$. We expect that the 
results of applying $\A$ to elements in $X$ does not lay in $Y$. 

Nevertheless, using the structure in \eqref{decomposition} and the 
smoothing properties we will be able to show by induction that if 
the error is in $Y$ at one step of the iteration, we can estimate the 
error in subsequent steps of the iteration. Note that the new error is 
the error in the Taylor approximation of $\X \circ(K + \Delta)$, 
which is the error in the Taylor approximation of $\N\circ(K + \Delta)$. 

Of course, we also need to ensure that the initial approximation satisfies 
this hypothesis. In the practical applications, we will just take a
trigonometric polynomial. 
\end{remark}

\subsection{Symplectic properties}

We will need that there is some exact symplectic structure. 
In our method, this does not play a very important role. 
We just use  the preservation of the  symplectic structure
to derive certain identities in the  (finite dimensional) 
center directions. These are called \emph{automatic 
reducibility} and use the exactness to show that some 
(finite dimensional) averages vanish (\emph{vanishing lemma})
so that we can prove the result without adjusting parameters. 

We will assume that there is a (exact)  symplectic form in 
the space $X$ and that the evolution equation
\eqref{eq:evolutionary-formal} 
can be written in Hamiltonian form in a suitable weak sense, which 
we will formulate now.

Motivated by the examples in Sections~\ref{applications}
and \ref{applications2} and others in the literature, 
 we will assume that the symplectic 
form is just a constant operator over the whole space $X$ 
(notice that we can identify all the tangent spaces). 
We will not consider the fact that the symplectic form depends
on the position. Note that heuristically, the fact that the 
symplectic form is constant ensures $d \Omega = 0$ and, because 
we are considering a Banach space, Poincar\'e lemma would give
$\Omega = d\alpha$. We will need only weak forms of these facts. General  symplectic forms in 
infinite dimensions may present surprising  phenomena not 
present in finite dimensions \cite{chernoffM74,bambusi99,KappelerP03}. 
Fortunately, we only need very few properties in 
finite dimensional subspaces in a very weak sense.

\medskip

{\bf H3} 
There is  an anti-symmetric bounded operator $\Omega: X\times X \rightarrow \complex$
taking real values on real vectors. 

The operator $\Omega$ is assumed to be non-degenerate in the sense that 
$\Omega(u, v) = 0 \ \forall v \in X$, implies $u = 0$. 

$\Omega$ will be refered to as \emph{the symplectic form}.

\medskip
As we mentioned above, we are assuming that the symplectic form is 
constant. 

In some of the applications, $\Omega$ could be a differential operator
or the inverse of a differential operator. 
When $\Omega$ is a differential operator, the fact that 
$\Omega$ is bounded only means that we are considering a space $X$ 
consisting of functions with high enough regularity. 
The form $\Omega$ could be unbounded in $L^2$ or in spaces consisting of functions 
with lower regularity than the functions in $X$. 

Notice that given a $C^1$ embedding $K$ of $\torus^\ell$ to $X$ we can define 
the pull-back of $\Omega$ by the customary formula
\begin{equation} \label{pullback} 
K^*\Omega_\th (a, b) = \Omega( DK(\th) a, DK(\th)b) 
\end{equation} 
The form $K^* \Omega$ is a form on $\torus^\ell$. If $K$ is $C^r$ as a mapping form $\torus^\ell$ to $X$ (in our applications it will be
analytic), the form $K^* \Omega$ will be $C^{r -1}$. 

\medskip

{\bf H3.1} We will assume that $\Omega$  is exact in the sense that, for all 
$C^2$ embeddings $K: \torus^\ell \rightarrow X$ we have 
\begin{equation}\label{weakexact} 
K^* \Omega = d \alpha_K 
\end{equation} 
with $\alpha_K$ a one-form on the torus. 

\medskip

In the applications we will have that $\alpha_K = K^* \alpha$ for 
some 1-form in $X$. Note that if $\Omega$ is not constant, we will need that $\alpha$ depends on the position. 

\medskip

{\bf H4} 
There is an analytic function $H:X \rightarrow \complex$ such 
that  for any $C^1$ path $\gamma: [0,1] \rightarrow X$, we 
have 
\begin{equation}\label{eq:Hamiltonian-form}
H(\gamma(1)) - H(\gamma(0)) = \int_0^1 \Omega(\X(\gamma(s)), \gamma'(s) )  \, ds 
\end{equation} 

Note that {\bf H4} is a weak form of the standard Hamilton equations
$i_\X \Omega = d H$. We take the Hamiltonian equations and 
integrate them along a path to obtain \eqref{eq:Hamiltonian-form}.

\medskip

A consequence of {\bf H3} and {\bf H4} is  we have that for 
any closed loop $\Gamma$ with image in $\TT^\ell$
\begin{equation}\label{vanishing-assumption} 
\int_{\Gamma} i_{\X \circ K} K^*\Omega=0.
\end{equation}

\begin{remark} 
The formulation of \eqref{eq:Hamiltonian-form} 
is a very weak version of the Hamilton equation. 
In particular, it is somewhat weaker than the formulation 
in \cite{Kuksin06}, but on the other hand, we will assume more 
hyperbolicity properties than in \cite{Kuksin06}. 
\end{remark} 

\subsubsection{Some remarks on the notation for the  symplectic form} 

The symplectic form can be written as
\[
\Omega(u,v) = \langle u, J v \rangle_Z 
\]
where $Z$  is a Hilbert space and 
$\langle \cdot, \cdot  \rangle_Z $ denotes the inner product in $Z$
and $J$ is a (possibly unbounded) operator in $Z$ -- but bounded  from $X$ to 
$Z$. 

Once we have defined the operator $J$, we can talk about  
the operator $J^{-1}$ if it is defined in some domain. 

The evolution equations can be written formally
\begin{equation} \label{hamiltonian} 
\frac{d u}{dt}  = J^{-1} \nabla H(u) 
\end{equation}
where $\nabla H$ is the gradient understood in the sense of the metric in $Z$. 
In the concrete applications here, we will take $Z = L^2$, $X = H^m$, 
$Y = H^{m-a}$ for large enough $m$. Of course, in well posed systems
we can take $X = Y$. 

We recall that the definition of a gradient (which is a vector field) requires 
a metric to identify differentials with vector fields. This is true even in finite 
dimensions. In infinite dimensions, there are several more subtleties such as the 
way that the derivative is to be understood.  Hence, we will not use 
much the gradient notation and the operator $J$ except in 
Section~\ref{sol-center}, which is finite dimensional.

\begin{remark}
In the Physical literature (and in the traditional calculus of 
variations) it is very common to take $Z$ to be always $L^2$, even if 
the functions in the space $X$ or $Y$ are significantly more differentiable. In some 
ways the space $Z = L^2$ is considered as fixed and the spaces $X, Y$ are mathematical choices. 
So that the association of the symplectic form to a symplectic operator is 
always done with a different inner product $Z$. 
The book \cite{Neuberger} contains a systematic treatment of the  use of gradients associated to 
Sobolev inner products. 
\end{remark}

\subsection{Diophantine properties}
\label{sec:Diophantine}
We will consider frequencies that satisfy the standard Diophantine properties. 

\begin{defi}\label{rotVF}
Given $\kappa>0$ and $\nu \geq \ell-1$, we define $\D(\kappa,\nu)$ 
as the set of frequency vectors $\omega \in \real^\ell$ satisfying the 
Diophantine  condition:
\begin{equation}
\label{eq:Diophantine}
|\omega \, \cdot \,k|^{-1} \leq \kappa |k|^{\nu},\,\,\,\,\,\,\mbox{for all $k\in\integer^\ell-\left \{ 0\right \}$}
\end{equation}
where $|k|=|k_1|+...+|k_\ell|$. We denote 
 $$\D(\nu) = \cup_{\kappa > 0} \D(\kappa, \nu). $$ 
\end{defi}  

It is well known that when $\nu > \ell$, the set 
$\D(\nu)$ has full Lebesgue measure.

\subsection{Spaces of analytic mappings from the torus} 
\label{sec:embeddings} 
We will denote $D_{\rho}$ the complex strip of width $\rho$, i.e.  
\begin{equation*}
D_{\rho}=\left \{ z\in \complex^\ell/\integer^\ell: \,\,|\mbox{Im}\,z_i| < \rho\,\,i=1,...,\ell\right \}.
\end{equation*}

We introduce the following $C^m$-norm
for $g$ with values in a Banach space $W$
\begin{equation*}
|g|_{C^m(\mathcal{B}),W}=\displaystyle{\sup_{0 \leq |k| \leq m� } }
\displaystyle{\sup_{ z \in \mathcal{B}}} \,||D^kg(z)||_W. 
\end{equation*} 

Let $\H$ be a Banach space and consider $\A_{\rho,\H}$ the set of continuous functions on
$\overline{D_{\rho}}$, analytic in $D_{\rho}$ with values in $\H$.     
We endow this space with the norm
\begin{equation*}
\|u\|_{\rho,\H}=\displaystyle{\sup_{z \in D_{\rho}}}\|u(z)\|_{\H}.
\end{equation*}

$(\A_{\rho,\H},\|\cdot\|_{\rho,\H})$ is
well known to be a Banach space.  Some particular cases which will be important for 
us are when the space $\H$ is a space of linear mappings (e.g. projections).

We will also need some norms for linear operators. Fix $\theta \in
D_\rho$ and consider $A(\theta)$ a
continuous linear operator from $\H_1$ into $\H_2$,
two Banach spaces. Then we define
$\|A\|_{\rho,\H_1,\H_2}$ as
\begin{equation*}
\|A\|_{\rho,\H_1,\H_2}=\displaystyle{\sup_{z \in D_{\rho}}}\|A(z)\|_{\L(\H_1,\H_2)},
\end{equation*}
where $\L(\H_1,\H_2)$ denotes the Banach space of
linear continuous maps from $\H_1$ into $\H_2$
endowed with the supremum norm. 

\begin{defi}
Let $\TT^\ell = \RR^\ell/\ZZ^\ell$ and
$f \in L^1(\torus^\ell,\mathcal H)$ where $\mathcal H$ is some Banach space. We denote $\avg(f)$ its average on the $\ell$-dimensional torus, i.e. 
\begin{equation*}
\avg (f)  =\int_{\torus^\ell}f(\th)\,d\th.
\end{equation*}
\end{defi}

\begin{remark}
Of course, in the previous definition, since $\mathcal H$ might be an infinite-dimensional space, 
the above integral, in principle,  has to be understood as a Dunford integral. 
Nevertheless, since we will consider rather smooth functions, 
it will agree with simple approaches such as Riemann integrals. 
\end{remark}

\subsection{Non-degeneracy assumptions}
This section is devoted to the non-degeneracy assumptions associated
to approximate solutions $K$ of \eqref{eq:embedding-approx}. We first deal with
the spectral non degeneracy conditions. The crucial quantity is the
linearization equation around a map $K$ given by
\begin{equation}\label{varNH}
\frac{d\Delta}{dt}=A (\th+ \omega t)\Delta,
\end{equation}
where $A(\th) = D(\X \circ K) (\th) $ is an operator mapping $X$ into $Y$.

Roughly, we want to assume that there is a splitting of the space into directions 
on which the evolution can be defined either forwards or backwards
and that the evolutions thus defined are smoothing. 
We anticipate that in Section~\ref{change-nondeg}, we will present 
other conditions that imply Definition~\ref{ND1}. We will just need to 
assume approximate versions of the invariance. 

\begin{defi}\label{ND1} {\bf Spectral non degeneracy} 

We will say that an embedding $K:D_\rho \to X$ is spectrally non degenerate if for every $\theta$ in $D_\rho$, we can find a
splitting 
\begin{equation}\label{splitting}
X=X_\th^s \oplus X_\th^c \oplus X_\th^u 
\end{equation} 
with associated bounded projection
$\Pi_{\th}^{s,c,u} \in \mathcal L(X,X)$ and where $X^{s,c,u}_\th$
are in such a way that: 

\begin{itemize}
\item {{\bf SD1}}
The mappings  $\theta \rightarrow \Pi^{s,u,c}_\theta$ are 
in $\A_{\rho, \L(X,X)}$ (in particular,  analytic).

\item {{\bf SD2}}
The space 
 $X^c_\th$ is finite dimensional with dimension $2\ell$. Furthermore the restriction of the operator $J$ to $X^c_\th$ denoted $J_c$ induces a symplectic form on $X^c_\th$ 
which is preserved by the evolution on $X^c_\th$ (see below). 
\item{{\bf SD3}}
We can find families of 
operators 
\begin{equation} 
\begin{split} 
& U_\theta^{s}(t) : Y^s_\theta  \rightarrow X^s_{\th + \omega t} 
 \quad t > 0 \\
& U_\theta^{u}(t) : Y^u_\theta  \rightarrow X^u_{\th + \omega t} 
 \quad t < 0 \\
& U_\theta^{c}(t) : Y^c_\theta  \rightarrow X^c_{\th + \omega t} 
 \quad t \in \real \\
\end{split} 
\end{equation} 
such that: 
\begin{itemize} 
\item {{\bf SD3.1}}
The families $U^{s,c,u}_\th$ are cocycles over the rotation of angle $\omega$
(cocycles are the natural generalization of semigroups 
for non-autonomous systems)
\medskip
\begin{equation}
U^{s,c,u}_{\th + \omega t}(t') U^{s,c,u}_{\th}(t) = 
U^{s,c,u}_{\th}(t + t')
\end{equation} 
\item{{ \bf SD3.2}} The operators $U^{s,c,u}_\th$ are smoothing in the time direction 
where they can be defined and they satisfy assumptions in the 
quantitative rates. 
\medskip
 There exist
 $\alpha_1,\alpha_2 \in [0,1)$, $\beta_1,\beta_2,\beta_3^+, \beta_3^->0$ and
 $C_h>0$ independent of $\theta$ such that the evolution 
operators are characterized by the following rate conditions:
\begin{equation}\label{rate1} 
 \|{U}^s_{\th}(t) ||_{\rho,Y,X} \leq {C}_h e^{-\beta_1 t}t ^{-\alpha_1},\qquad t >  0,
\end{equation}
\begin{equation}\label{rate2}
\|{U}^u_{\th}(t) \|_{\rho,Y,X} \leq {C}_h{e^{\beta_2  t}}{|t|^{-\alpha_2}},
\qquad t <   0,
\end{equation}
\begin{equation}\label{rate3}
\begin{split}
& \|{U}^c_{\th}(t) \|_{\rho,X,X} \leq {C}_h e^{{\beta}_3^+ t},
\qquad t > 0  \\
& \|{U}^c_{\th}(t) \|_{\rho,X,X} \leq {C}_h e^{{\beta}_3^- |t|},
\qquad t <  0 
\end{split}
\end{equation} 
with $\beta_1 >\beta_3^+ $ and $\beta_2 >\beta_3^- $.

\item {{\bf SD3.3}}The operators $U^{s,u, c}_\th$ are fundamental solutions of the 
variational equations in the sense that
\begin{equation} \label{variational-1}
\begin{split} 
U^s_\theta(t) &= Id + \int_0^t  A(\theta - \omega \sigma ) U^{s}_{\theta-\omega\sigma}(\sigma)\, d\sigma \quad t > 0\\
U^u_\theta(t) &= Id +  
\int_0^t A(\theta - \omega \sigma ) U^{u}_{\theta-\omega\sigma}(\sigma)\, d \sigma
 \quad t < 0\\
U^c_\theta(t) &= Id +  
\int_0^t A(\theta - \omega \sigma ) U^{c}_{\theta-\omega\sigma}(
\sigma) \, d\sigma 
\quad t \in {\mathbb R} 
\end{split} 
\end{equation} 

\end{itemize}
\end{itemize} 
\end{defi}

\begin{remark} 
Note that as  consequence of the integral equations \label{variational} 
and the rate conditions \eqref{rate1}, \eqref{rate2}, \eqref{rate3} 
we have, using just the triangle inequality
\[
|| U_\th^s(t)||_{\rho, Y,Y} \le 1 + \int_0^t A s^{-\alpha_1} e^{-\beta_1} s \, ds
\]
Proceeding similarly for the others, we obtain 
\begin{equation}\label{newratesY}
\begin{split} 
& \|{U}^s_{\th}(t) ||_{\rho,Y,Y} \leq {\tilde C}_h e^{-\beta_1 t}\qquad t >  0,\\
& \|{U}^u_{\th}(t) \|_{\rho,Y,Y} \leq {\tilde C}_h{e^{\beta_2  t}},
\qquad t <   0, \\
& \|{U}^c_{\th}(t) \|_{\rho,Y,Y} \leq {\tilde C}_h e^{{\beta}_3^+ t},
\qquad t > 0  \\
& \|{U}^c_{\th}(t) \|_{\rho,Y,Y} \leq {\tilde C}_h e^{{\beta}_3^- |t|},
\qquad t <  0 \\
\end{split}
\end{equation} 
\end{remark} 

\begin{remark} 
We are not aware of any general  argument that would show 
that:
\begin{equation}\label{newratesX}
\begin{split} 
& \|{U}^s_{\th}(t) ||_{\rho,X,X} \leq {\tilde C}_h e^{-\beta_1 t}\qquad t >  0,\\
& \|{U}^u_{\th}(t) \|_{\rho,X,X} \leq {\tilde C}_h{e^{\beta_2  t}},
\qquad t <   0, \\
& \|{U}^c_{\th}(t) \|_{\rho,X,X} \leq {\tilde C}_h e^{{\beta}_3^+ t},
\qquad t > 0  \\
& \|{U}^c_{\th}(t) \|_{\rho,X,X} \leq {\tilde C}_h e^{{\beta}_3^- |t|},
\qquad t <  0 \\
\end{split}
\end{equation}
follow from the other assumptions. Needless to say, we would be happy to hear about one. 

One can, however, clearly have that
since $|| U_\th ^s(t)||_{X,X} \le || U_\th ^s(t) ||_{Y,X}$ so 
that the semigroups are exponentially decreasing for large $t$. 

One notable case, which happens  in practice, when 
one can deduce \eqref{newratesX} is when the spaces X and Y are 
Hilbert spaces. In such a case, taking  Hilbert space 
adjoints in \eqref{variational}
we obtain:
\[
U^s_\th(t)^* = Id + \int_0^t U^{s}_{\theta-\omega\sigma}(\sigma)^* 
 A(\th - \omega \sigma )^* \, d\sigma \quad t > 0
\]
and using the fact that the  adjoints preserve the norm, we can easily 
obtain the bounds in the same way as \eqref{newratesY}.
\end{remark} 

\begin{remark}\label{hamiltoniancase} 
We remark that when the equation preserves a symplectic  structure, 
we can have without loss of 
generality 
\begin{equation} 
\label{eq:symmetryexponent} 
\beta_3^+ = \beta_3^-, \qquad \beta_1 = \beta_2. 
\end{equation}  Conversely, if \eqref{eq:symmetryexponent} is 
satisfied, the center direction automatically preserves a 
symplectic structure.  See Lemma~\ref{omegadeg}. 

We anticipate that the results in Section~\ref{change-nondeg}
on persistence of trichotomies (a fortiori dichotomies) with smoothing 
are developed without assuming that the equation is Hamiltonian and, hence apply also 
to dissipative equations.  Similarly, the solutions of linearized equations 
in the hyperbolic directions developed in Section~\ref{sec:linearized-hyperbolic} are obtained without  using the Hamiltonian structure. 
The Hamiltonian structure is used only to deal with the linearized equations 
in the center direction in Section~\ref{sol-center}. 
\end{remark}

Let us comment on  the previous spectral non-degeneracy conditions.

The first observation is
that, if we assume 
that the spaces $X,Y$ are 
Sobolev spaces of high enough index (so that the functions in them 
are $C^r$ for $r$ high enough) 
then we have that \eqref{variational-1} holds in a classical sense if it holds in 
the sense of mild solutions (the sense of integral equations). In the 
applications we have in mind, it is always possible to take the 
spaces $X$, $Y$ that have arbitrarily high derivatives. 

Then, \eqref{variational-1} is just a 
form of 
\begin{equation} \label{variational-mild}
\begin{split} 
\frac{d}{dt} U^s_\theta(t) &= A(\theta + \omega t ) U^{s}_{\theta}(t) \quad t > 0\\
\frac{d}{dt} U^u_\theta(t) &= A(\theta + \omega t ) U^{}_{\theta}(t) \quad t < 0\\
\frac{d}{dt} U^c_\theta(t) &= A(\theta + \omega t ) U^{c}_{\theta}(t) \quad t \in \real 
\end{split} 
\end{equation} 
Often \eqref{variational-1} is described as saying that the derivatives in 
\eqref{variational-mild} are understood in the mild sense. 

Making sense of the integrals in \eqref{variational-1}
is immediate after some reflection. 
Our conditions just require the existence of an 
evolution for positive and negative times on certain subspaces.  
The important conditions on these evolutions are
the  characterization of the splitting by rates 
\eqref{rate1}-\eqref{rate2}, expressing the fact that the operators are
 bounded and smoothing from $Y$ into $X$ 
(recall that $X \hookrightarrow Y$). 
If the system were autonomous, such properties would hold under 
some spectral assumptions on the operator
 $A(\theta)$ (bisectoriality or generation of strongly 
continuous semi-groups, see \cite{pazy}).

Since  the spaces $X^c_\th$ and $Y^c_\th$ are 
finite dimensionals 
and of the same dimension, 
the evolution $U^c_\th(t)$ can be considered as 
an operator from  $Y_\th^c$ to $Y^c_{\th + t \omega}$. 

In the finite dimensional case (or in the cases where there is 
a well defined evolution), property {\bf SD.1} follows from the 
contraction rates  assumption {\bf SD.3} by a fixed point 
argument in spaces of analytic functions. 
See \cite{HaroL06}. In our case, we have not been 
able to adapt the finite dimensional argument, that is 
why we have included it as an independent assumption (even if 
may end up be redundant). 
We note that {\bf SD.1}, {\bf SD.3} are 
 clearly true when $\mathcal N \equiv 0$ 
and in this paper we  will show it is stable under perturbations, hence
{\bf SD.3} will hold for all small enough  $u$. This suffices for 
our purposes, so we will not pursue the question of whether 
{\bf SD.1} can be obtained from {\bf SD.3} in general. 

The fact that $\Omega |_{X_\th^c} $ is non-dengenerate (which is a part 
of {\bf SD.2}) follows from the rate conditions {\bf SD.3} 
as we show in Lemma~\ref{omegadeg}.

One situation when all the above abstract properties are satisfied is
when the evolution is given just by the linear part $\A$, i.e.  $\N
\equiv 0$.  The assumptions of our set up are verified if the spectrum
of $\A$ is just eigenvalues of finite multiplicity and the spectrum is
the union of a sector around the positive axis, another sector around
the negative axis and a finite set of eigenvalues of finite
multiplicity around the imaginary axis.  Then, the stable space is the
spectral projection over the sector in the negative real axis, the
unstable space will be the spectral projection over the sector along
the positive axis and the center directions will be the spectral space
associated to the eigenvalues in the finite set.  There are many
examples of linear operators satisfying these properties.

It will be important
that  the main result of Section \ref{change-nondeg} is 
these structures persist when we add a lower order perturbation which
is small enough. Indeed, we will show that if we 
find splittings that satisfy them approximately enough, there is
true splitting nearby. This would allow to validate numerical 
computations, formal expansions, etc.

\subsubsection{The twist condition}
As it is standard in KAM theory, one has to impose another
non-degeneracy assumption, namely the twist condition. This is the
object of the next definition. Notice that it amounts to 
a finite dimensional matrix being invertible. It is 
identical to the conditions that were used in 
the finite dimensional cases \cite{LGJV05,fontichdelLS071}.

\begin{defi}\label{ND2}
Denote 
  $N(\th)$ 
the $\ell \times \ell$ matrix 
such that $N(\th)^{-1}=DK(\th)^\perp DK(\th) $ 

Denote $P(\th)=DK(\th)N(\th)$

Let  $J_c$ stand for  restriction of symplectic operator $J$ to 
$X^c_\th$. We will show in Lemma~\ref{omegadeg} 
that the form $\Omega_c \equiv \Omega|_{X_\th^c}$ is non-degenerate so
that the operator $J_c$ is invertible.

We now define the twist matrix $S(\th)$ (the motivation will become aparent 
in Section~\ref{sol-center}, but it is identical to the definition in 
the finite dimensional case in \cite{LGJV05,fontichdelLS071}). The average of the matrix 

\begin{equation}
\label{twistdefined}
S(\th)=N(\th) DK(\th)^\perp [J_c^{-1}\partial_\omega(DK\,N)-A J_c^{-1}(DK\,N)](\th)
\end{equation}
is non-singular.

\end{defi}  

We note that the  matrix $S$ in \eqref{twistdefined} 
is a very explicit expression 
that can be computed out of the approximate solution of 
the invariant equation and the invariant bundles just taking 
derivatives, projections and performing algebraic operations.
So that it is easy to verify in applications when we are given 
an approximate solution.

We will say that an embedding is non-degenerate (and we denote it $K
\in ND(\rho)$) if it is non-degenerate in the sense of Definitions
\ref{ND1} and \ref{ND2}.

\begin{remark} 
As it will become apparent in the proof, the twist condition has a 
very clear geometric meaning, namely that the frequency of 
the quasiperiodic motions changes when we change the initial conditions in 
a direction (conjugate to the tangent to the torus). 

Note that, given 
an invariant torus, we can consider it as an approximate solution 
for similar frequencies and that the twist condition also 
holds. 

Using the a-posteriori theorem shows that under the conditions, we
have many tori with similar frequencies near to the torus.
\end{remark}

\subsubsection{Description of the iterative step}

Once the two non-degeneracy conditions are met for the initial guess
of the modified Newton method, the iterative step  goes as follows:  

\begin{enumerate} 
\item We project the cohomological equations with respect to the
  invariant splitting. 
\item We then solve the equations for the stable and unstable subspaces. 
\item We then solve the equation on the center subspace. This involves small divisor equations. 
We note that solving the equation in the center requires to use the exactness 
so that we can show that the equations are solvable. 
\item To be able to iterate we will need to show that the corrections also satisfy the non-degeneracy conditions (with only some slightly worse quantitative
assumptions). This amounts to showing the stability of the 
spectral non-degeneracy conditions, and developing explicit estimates of the 
changes in the properties given the changes on the embedding. 
\end{enumerate}

\subsection{Statement of the results}
\label{sec:statement} 

\subsubsection{General abstract results} 
\label{sec:statement-general}

The following Theorem~\ref{existence} 
 is the main result of this paper. It provides
the existence of an embedding $K$ for equation \eqref{eq:embedding}
under some non-degeneracy conditions for the initial guess.  We stress
here  that Theorem~\ref{existence} is in an {\sl a posteriori}
format (an approximate solution satisfying 
nondegeneracy conditions implies 
the existence of  a true solution close to it). As already pointed out in the papers
\cite{fontichdelLS07a,fontichdelLS071,FontichLS15}, this format
allows to validate many methods that construct approximate solutions, 
including asymptotic expansions or numerical solutions. 
We also note that it has several automatic consequences 
presented in Section~\ref{sec:consequences}.

\begin{thm}\label{existence}
Suppose assumptions ${\bf H1, H2,H3}$ are met; let $\omega \in \D(\kappa,\nu)$ for some $\kappa>0$ and $\nu\ge \ell-1$. Assume that
\begin{itemize}
\item $K_0$ satisfies the non-degeneracy Conditions
  \ref{ND1} and \ref{ND2} for some $\rho_0 >0$. 

\item
We assume that the range of $K_0$ acting on a complex extension of 
the torus is well inside of $\U$ the domain of analyticity of $\N$ introduced in 
{\bf H2}. More precisely: 

\begin{equation*} 
\dist_X( K_0(D_\rho),  X \setminus \U) \ge r > 0
\end{equation*}   
That is, if $x = K_0(\th)$, 
$\th \in D_{\rho_0}$ and $|| x -y||_{X} \le \rho_0$, then $y \in \U$. 

\end{itemize} 
Define the initial error
\[
E_0=\partial_\omega K_0 - \X\circ K_0
\]
Then there exists a constant $C>0$ depending on $l$, $\nu$,
$\rho_0$, $|\X|_{C^1(B_r)}$, $\|DK_0\|_{\rho_0,X}$,
$\|N_0\|_{\rho_0}$, 
$\|S_0\|_{\rho_0}$, 
(where $S_0$ and $N_0$ are as
in Definition \ref{ND2} replacing $K$ by $K_0$) and the norms of the
projections $\|\Pi^{c,s,u}_{K_0(\th)}\|_{\rho_0,Y,Y}$ such that, if $E_0$ satisfies the estimates
\begin{equation*}
C |\avg (S_0)^{-1}|^2\kappa^4 \delta^{-4\nu} \|E_0\|_{\rho_0,Y} <1
\end{equation*}
and 
\begin{equation*}
C |\avg (S_0)^{-1}|^2 \kappa^2 \delta^{-2\nu} \|E_0\|_{\rho_0,Y} <r,
\end{equation*}
where $0 < \delta \le \min(1,\rho_0/12)$ is fixed, then there exists an
embedding $K_{\infty} \in ND(\rho_{\infty}:=\rho_0-6\delta)$ such that  
\begin{equation}\label{transSolNH2}
\partial_{\omega} K_{\infty}(\th)= \X \circ K_{\infty}(\th)).
\end{equation}
Furthermore, we have the estimate   
\begin{equation}
\label{changestimate} 
\|K_{\infty}-K_0\|_{\rho_{\infty},X} \leq C |\avg (S_0)^{-1}|^2\kappa^{2} \delta^{-2\nu} \|E_0\|_{\rho_0,Y}. 
\end{equation}

The torus $K_\infty$ is also spectraly non degenerate in the sense of
Definition~\ref{ND1} with $\rho$ in Definition~\ref{ND1} replaced 
by $\rho_\infty$ and with  other constants differing
from those of $K_0$ modifying by an amount bounded by 
$C \| E_0 \|_{\rho_0}$.

Furthermore, if we have two solutions $K_1, K_2$ satisfying 
\eqref{eq:embedding} and spectrally nondegenerate in the 
sense of Definition~\ref{ND1} and that  satisfy
\begin{equation} \label{uniquenesshypothesis} 
\|K_1-K_2 \|_{\rho_{\infty},X} \leq C |\avg (S_0)^{-1}|^2 \kappa^{2} \delta^{-2\nu} 
\end{equation} 
Then, there exists $\sigma \in \real^{\ell}$ such that 
\begin{equation} 
\label{uniquenessconclusion} 
K_1(\theta) = K_2(\theta + \sigma)
\end{equation} 
\end{thm} 

The  statement  that $K_\infty$ satisfies the 
Definition~\ref{ND1} is a consequence of the 
estimates in Section~\ref{change-nondeg}. 

The uniqueness statement will be proved in 
Section~\ref{sec:uniqueness}. It is exactly the same 
as the one in the finite dimensional case in \cite{fontichdelLS071}. 

\subsubsection{Some consequences of the a-posteriori format}
\label{sec:consequences}
The a-posteriori format leads inmediately to several consequences. When we have systems that depend on parameters, observing that the 
solution for a value of the parameter is an approximate solution for 
similar values of the parameters, one obtains \emph{Lipschitz dependence on 
parameters, including the frequency}. 

If one can obtain Lindstedt expansions in the parameters, one can obtain 
Taylor expansions. If the parameter ranges over $\real^n$, this is 
the hypothesis of the converse Taylor theorem \cite{AbrahamR67,Nelson69}
so that one obtains \emph{smooth dependence on parameters}. 
In the case that the parameters range on a closed set, we obtain one 
of the conditions of the Whitney extension theorem. Some general 
treatments are \cite{Vano02,CallejaCL15}.

In many perturbative solutions, one gets that the twist condition is 
small but that the error is much smaller. Note that in the main result,
we presented explictly that the smallness conditions on the error 
are proportional to the square of the twist condition. Hence, we
obtain the \emph{small twist condition}. Note also that the twist condition 
required is not a global condition on the map, but rather a condition 
that is computed on the approximate solution.  Indeed, we will take 
advantage of this feature in the sections on applications. 

The abstract theorem can be applied to several spaces. 
Some spaces of low regularity (e.g. $H^m$) and others with high regularity
(e.g. analytic). The existence results are more powerful in the 
high regularity spaces and the local uniqueness is more 
powerful in the low regularity spaces. 

Given a sufficiently regular solution, one can obtain an analytic 
approximate solution by truncating the Fourier series, which leads to 
an analytic solution, which has to be the original one. Hence, 
one can \emph{bootstrap the regularity}.  See \cite{CallejaL10} for 
an abstract version.

\subsubsection{Results for concrete equations} 
\label{sec:results-particular} 

Consider the following one-dimensional Boussinesq equation subject to periodic boundary conditions, i.e. 

\begin{equation}\label{bou}
u_{tt}= \mu u_{xxxx}+u_{xx}+(u^2)_{xx}
\,\,\,\, x \in \mathbb{T},\,\,t\in \RR. 
\end{equation} 
Looking for solutions of 
the linearization of the form  $u(x, t) = e^{2 \pi i ( k x +  \omega(k) t )}$
we obtain the eigenvalue relation 
\begin{equation}
\label{dispersion} 
\omega^2(k) = - \mu |k|^4 (2 \pi)^2 +  |k|^2 
\end{equation} 
We see that for large $|k|$, $\omega(k) \approx \pm 2i\pi  \mu^{1/2} |k|^2$. 
Hence, the Fourier modes may grow at an exponential rate and 
 the rate is quadratic in the index of the mode. So 
that even analytic functions evolving under the 
linearized equation leave instaneously even spaces of 
distributions.  The non-linear 
term does not restore the well posedness. 
(See Remark~\ref{illposed-nonlinear}.) The previous equation \eqref{bou} is Hamiltonian on $L^2(\TT)$. Indeed, we introduce first the skew-symmetric operator 
$$J^{-1}=\begin{pmatrix} 0  & \partial_x \\ \partial_x & 0 \end{pmatrix}$$

and define 
$$H_{\mu}(u,v)=\int_{0}^1 \frac12 \Big \{ u^2+v^2-\mu (\partial_x u)^2 \Big \}+\frac{1}{3} u^3. $$

Therefore, equation \eqref{bou} writes 
\begin{equation} \label{hamilBou}
\dot{z}=J^{-1} \nabla H_{\mu}(z),\,\,z=(u,v)
\end{equation}
where $\nabla $ has to be understood w.r.t. the inner product in $L^2(\TT)$. Note, however that when $\mu$ is small enough, there are several
values of $k$, which for which $\omega(k)$ is real. We denote by
$\omega^0$ the vector whose components are all the real frequencies
that appear

\begin{equation}\label{omega0bou} 
\begin{split}
&\omega^0 = (\omega(k_1), \omega(k_2), \ldots, \omega(k_\ell));
\\
&\{k_1,\ldots k_\ell\} = \{k \in \ZZ \, | \, k > 0 ; 
- \mu |k|^4 (2 \pi)^2 +  |k|^2 \ge 0 \}
\end{split} 
\end{equation}
We can think of $\omega^0$ as the frequency vector of the motions for 
very small amplitude. 

Note that  the equation \eqref{bou}  conserves
the 
quantity
$\int_0^1 \partial_t u(t,x) dx$ (called the momentum).
Hence $\int_0^1 u(t,x) dx$ (the center of mass)  evolves linearly in 
time. 

We can always change to a system of coordinates in which 
$\int_0^1 \partial_t u(t,x) dx = 0 $. Hence, in this system 
$\int_0^1 u(t,x) dx = cte$. By adding the constant we can assume 
without loss of generality that $\int_0^1 u(t,x) \, dx = 0$. 

Hence we will assume (without loss of generality) 
that 

\begin{equation}\label{automaticsymetry} 
\begin{split} 
& \int_0^1 \partial_t u(t,x) dx = 0 \\
& \int_0^1 u(t,x) dx = 0 \\
\end{split} 
\end{equation}
\begin{remark}\label{equationcount} 
We emphasize that the two parts of \eqref{automaticsymetry} are 
not two independent equations. The first one is just a derivative 
with respect to time of the second. Even if the relation is formal, 
it makes sense when we are dealing with polynomial approximate solutions. 
\end{remark}

We also note that the  equation \eqref{bou} leaves invariant 
the space of functions which are symmetric around $x$ (it does not 
leave invariant the space of functions antisymmetric around $x$). 
Hence, we can consider the equation as defined on the space of 
general functions or in the space of symmetric functions. 
\begin{equation}\label{symmetry} 
u(t,x) = u(t,-x)
\end{equation} 

The main difference between the symmetric and the general case is that 
center space is of different dimension.

We introduce the following Sobolev-type spaces $H^{\rho,m}(\TT)$ for $\rho >0$ and $m \in \NN$ being the space of analytic functions $f$ in $D_\rho$ such that the quantity 

$$\|f\|^2_{\rho,m}=\sum_{k \in \ZZ} |f_k|^2 e^{4\pi \rho |k|}(|k|^{2m}+1)$$
is finite, and where $\left \{ f_k \right \}_{k \in \ZZ}$ are the Fourier coefficients of $f$. Let 
\begin{equation} 
\label{Xbou}
X=H^{\rho, m}(\TT) \times H^{\rho,m-2}(\TT) 
\end{equation}
for $m \geq 2$. 

We state the following conjecture. 
\begin{conj}\label{conj}
Consider a parameter $\mu>0$ in \eqref{bou} such that the center 
space has dimension $2\ell \geq 2$ and fix a Diophantine exponent $\nu >\ell$ , a regularity exponent $m>5/2$ and a positive analyticity radius $\rho_0 $.

Then there exist three explicit functions $a, b_d, b_a : \RR^+ \rightarrow \RR^+$ such that 
$$a(s) \to 0,b_d(s), b_a(s)  \to \infty, \,\, s \to 0$$
 in such 
a way that: for $\ep$ sufficiently small, denote by 
$B_{a(\ep)}(\omega^0) \subset \RR^\ell $   the ball of radius 
$a(\ep)$ around $\omega^0$ and let $\omega \in 
\D(b(\ep),\nu) \cap B_{a(\ep)}(\omega^0)$. 

Then, there  exists $K$, an analytic function from $D_{\rho_0}
\to H^{\rho, m}(\TT) \times H^{\rho,m-2}(\TT) $ solving 
\eqref{eq:embedding} with frequency $\omega$. 

Furthermore, 
\[
\frac{ | \mathcal D(b(\ep),\nu) \cap B_{a(\ep)}(\omega^0)|}
{|  B_{a(\ep)}(\omega^0)|}
 \rightarrow 1
\]

The mapping that given $\omega$ produces $K$ is Lipschitz
when $K$ are given the topology of analytic embeddings 
from $D_{\rho'}$ to $X$ when $\rho' < \rho_0$. 
\end{conj}

In Section~\ref{sec:Lindstedt} 
we present a complete proof of the following result. 

\begin{thm}\label{thBou}
Conjecture \ref{conj} is true under the extra assumption that $\ell=1$ which amounts to take $\mu \in [\frac{1}{8\pi},\frac{1}{2\pi})$. 
\end{thm}

Informally, following the standard Lindstedt procedure, 
for $\ep$ small we find families of approximate 
solutions up to an error which is smaller than an 
arbitrarily large  power of $\ep$. 

We can also verify that the non-degeneracy assumptions hold 
with a condition number which is a fixed power of $\ep$.  
If $\ep$ is very small one can allow frequencies with a large 
Diophantine constant, 
and obtain that the functions are analytic in a very large
domain.  
As we will see in the proof, we can take the functions
$a, b_d, b_a$ to be just powers. 

The first step of constructing very approximate solutions 
is accomplished for all values of $\mu$ as in Conjecture~\ref{conj}. 

To verify the non-degeneracy conditions, it suffices to compute 
the determinant of an explicit matrix and checking it is not zero. 
This is the only step we are missing to verify Conjecture~\ref{conj}. 
This calculation is, not very hard, but it is tedious. Of course, 
there may be insights that make it possible to verify it. In the present paper, we will concentrate on $\ell=1$ to check this condition.

\begin{remark} 
We expect that Theorem~\ref{thBou} can be greatly expanded 
(a wider range of parameters, removing  the symmetry conditions)
by just performing longer calculations using the Lindstedt method.
We hope to come back to this problem in future work. 
\end{remark} 

\begin{remark}
Note that the case $\ell = 1$, amounts to periodic orbits so that there
are no small denominators. In this case, one can use simpler 
fixed point theorems. There are already numerical computer
assisted proofs in this case \cite{CastelliGL15,FiguerasGLL15}. 
\end{remark}

 Similar results will be proved for
other equations such as the Boussinesq system of water waves (see
Section \ref{applications2}). The system under consideration is

\begin{equation}\label{systemBou}
\partial_t    
\begin{pmatrix} 
u  \\
v \\
\end{pmatrix}
=
\begin{pmatrix} 
0 & -\partial_x -\mu \partial_{xxx}  \\
-\partial_x & 0 \\
\end{pmatrix}
\begin{pmatrix} 
u \\
v \\
\end{pmatrix}
+
\begin{pmatrix} 
\partial_x (uv) \\
0 \\
\end{pmatrix}
\end{equation}
where $t>0$ and $x \in \TT.$ System \eqref{systemBou} 
has a Hamiltonian structure given by:

$$J^{-1}=\begin{pmatrix} 0  & \partial_x \\ \partial_x & 0 \end{pmatrix}$$

and 
$$H_{\mu}(u,v)=\int_{0}^1 \frac12 \Big \{ u^2+v^2-\mu (\partial_x v)^2 \Big \}+\int_{0}^1 uv^2. $$

In this case, one has to take 
$$X=H^{\rho,m}(\TT) \times H^{\rho, m+1}(\TT)$$
and 
$$Y=H^{\rho,m-1}(\TT) \times H^{\rho, m}(\TT). $$

The elementary linear analysis around the $(0,0)$ equilibrium has been performed in \cite{llave08}. The dispersion relation is given by
\begin{equation}  
\omega(k)  = \pm |k| 2 \pi i
\sqrt{ 1 - 4 \pi^2 \mu k^2} \quad 
k \in \integer 
\end{equation} 
We take the principal determination of 
the square root .
We denote by $\omega^0$ the vector whose components are
all the real frequencies that appear
\begin{equation}\label{omega0first} 
\begin{split}
&\omega^0 = (\omega(k_1), \omega(k_2), \ldots, \omega(k_\ell));
\\
&\{k_1,\ldots k_\ell\} = \{k \in \ZZ \, | \, k > 0 ; 
  1 - 4 \pi^2 \mu k^2 \ge 0 \}
\end{split} 
\end{equation}

Similarly to the Boussinesq equation, we state
\begin{conj}\label{thBouWW}
Fix a Diophantine exponent $\nu >\ell$ and a regularity exponent $m$ large. Then there exist three explicit functions $a, b_d, b_a : \RR^+ \rightarrow \RR^+$ such that 

$$a(s) \to 0,b_d(s), b_a(s)  \to \infty, \,\, s \to 0$$
 in such 
a way that: for $\ep$ sufficiently small, denote by 
$B_{a(\ep)}(\omega^0) \subset \RR^\ell $   the ball of radius 
$a(\ep)$ around $\omega^0$ and let $\omega \in 
\D(b(\ep),\nu) \cap B_{a(\ep)}(\omega^0)$.

Then, there  exists $K \in X$ solving 
\eqref{eq:embedding} with frequency $\omega$, the parametrization of the Boussinesq system for water waves \eqref{systemBouTemp}. 

Furthermore, 
\[
\frac{ | D_h(b(\ep),\nu) \cap B_{a(\ep)}(\omega^0)|}
{|  B_{a(\ep)}(\omega^0)|}
 \rightarrow 1
\]
\end{conj}

\begin{thm}\label{thBouSystem}
Conjecture \eqref{thBouWW} is true provided that $\ell=1$, i.e. $\mu \in [\frac{1}{8\pi},\frac{1}{2\pi})$. 
\end{thm}

\section{The linearized invariance equation}
\label{sec:linearized}

The crucial ingredient of the Newton method is to solve the linearized
operator around an embedding $K$. This is motivated because one can hope to
improve the solution of \eqref{eq:evolutionary-formal}.  Notice the appearence of
the linearized evolution does not have a dynamical motivation. The
linearized equation does not appear as measuring the change of the
evolution with respect to the initial conditions, it appears as the
linearization of \eqref{eq:evolutionary-formal}.

Let us denote 
\begin{equation}\label{operatorinvariance} 
\mathcal F_\omega(K)=\partial_\omega K- \X \circ K.
\end{equation}

Clearly, the invariance equation \eqref{eq:embedding} can be 
written concisely as $\mathcal F_\omega(K) = 0$.

We prove the following result. 
\begin{lemma}\label{main}
Consider the linearized equation
\begin{equation}\label{lin-eq}
D\mathcal{F}_\omega(K)\Delta =-E. 
\end{equation}
Then there exists a  constant $C$ that depends on $\nu$, $l$,
$\|DK\|_{\rho,X}$,  $\|N\|_\rho$, $\|\Pi^{s,c,u}_{\th}\|_{\rho,Y,Y}$, $|(\avg(S))^{-1}|$
and the hyperbolicity
constants such that assuming 
that  $\delta \in (0, \rho/2)$ satisfies
\begin{equation} \label{smallness1} 
C \kappa \delta^{-(\nu + 1)} \| E\|_{\rho,Y} < 1
\end{equation} 
we have
\begin{itemize}
\item[A] There exists an approximate solution $\Delta $
of \eqref{lin-eq},  in the
following sense: there exits a function $\tilde{E}(\th)$ such that
$\Delta $ solves exactly
\begin{equation}\label{approximatesol} 
\begin{split} 
& D_{K}\mathcal{F}_\omega(K)\Delta=-E+\tilde{E}\\
\end{split} 
\end{equation}
with the following estimates: for all $\delta \in (0,\rho/2)$  
\begin{equation}\label{improve1}
\begin{split}
&\|\tilde{E}\|_{\rho-\delta,Y} \leq C \kappa^2 \delta^{-(2\nu+1)} 
\|E\|_{\rho} \|\mathcal{F}_\omega(K)\|_{\rho,Y}\\
& \|\Delta\|_{\rho-2\delta,X} \leq C \kappa ^2 \delta ^{-2\nu}
\|E\|_{\rho,Y},  \\
& \|D \Delta\|_{\rho-2\delta,X} \leq C \kappa ^2 \delta ^{-2\nu -1}
\|E\|_{\rho,Y},  
\end{split} 
\end{equation}
\item[B]
 If $\Delta_1$ and $\Delta_2$ are approximate solutions of 
the  linearized equation
\eqref{lin-eq}
  in the 
sense  of \eqref{approximatesol},  then there exists $\alpha \in
  \mathbb{R}^\ell$ such that for all $\delta \in (0,\rho)$
\begin{equation}\label{estimAvg}
\|\Delta_1- \Delta_2-DK(\th)\alpha\|_{\rho-\delta,X} \leq C \kappa
^2 \delta ^{-(2\nu+1)} \|E\|_{\rho,Y}
\|\mathcal{F}_\omega(K)\|_{\rho,Y}. 
\end{equation}
\end{itemize}

\end{lemma}

The previous Lemma~\ref{main} 
is the cornerstone of the KAM iteration and the goal of the following
sections is to prove this result. We will also need to prove that the
non-degeneracy conditions are preserved under the iteration and that
the constants measuring the non-degeneracy deteriorate only 
slightly. This will follow from the quantitative estimates developed 
in Section~\ref{change-nondeg}.

Note that \eqref{approximatesol}, \eqref{improve1} is the main ingredient of 
several abstract implicit function theorems which lead to the existence of
a solution. See, for example \cite{Zehnder75} or, particularly
\cite[Appendix A]{CallejaL} for implicit function theorems based on
existence of approximate inverses with tame bounds. 

Note also that in part (2) of Lemma~\ref{main} we have established some 
uniqueness for the solutions of the linearized equation. In 
Section~\ref{sec:uniqueness} we will show how this can be used to prove 
rather directly the uniqueness result in Theorem~\ref{main}. 

The proof of Lemma~\ref{main} is based on decomposing the 
equation into equations along the invariant bundles 
assumed to exist in the hypothesis that the approximate solution 
satisfies Definition~\ref{ND1}.  In the hyperbolic directions 
we will roughly use the variations of parameters formula, but we will have
to deal with the fact that the perturbations are unbounded. In the center 
directions, we  will have to use the number theory and the geometry. 
Fortunately, the center space is finite dimensional. 

The theory of solutions of the linearized equation is developed in 
Sections~\ref{sec:linearized-hyperbolic} 
and~\ref{sol-center} and Lemma~\ref{main} is 
obtained just putting together Lemma~\ref{computational} 
and the results in Section \ref{reducedEq}. 

We also note that the solutions of the linearized equation in the hyperbolic
directions will be important in the perturbattion theory of 
the  bundles, which is needed to show that the linearized equation can be 
applied repeatedly.

For coherence of the presentation, we have written together all the 
results requiring hyperbolic technology (the solution in the 
hyperbolic directions and the perturbation theory of bundles). 
Of course, we hope that the sections can be read independently
in the order prefered by the reader.  

\section{Solutions of linearized equations on the 
stable and unstable directions}
\label{sec:linearized-hyperbolic}

In this  Section we develop the study of linearized equations
of a system with splitting. See Lemma~\ref{computational}. 
Lemma~\ref{computational} will be one of the ingredients in 
Lemma~\ref{main}.

\begin{lemma}\label{computational}
Let $\mathcal L (\th):X \to Y$ for fixed $\th \in D_\rho$ be a vector field admitting an invariant splitting in the following sense: the space $X$ has
 an analytic family of splittings
$$
X = X^s_\th \oplus X^c_\th \oplus X^u_\th
$$
(We say that a splitting is analytic when the associated projections 
depend on $\th$ in an analytic way) 
invariant in the following sense: we can find families of operators
$\{ U_\th^s(t)\}_{t > 0} $, $\{U_\th^c(t)\}_{t \in \real}$ , $\{U_\th^u(t)\}_{t < 0}$ with domains
$X_\th^s$, $X^c_\th$, $X^u_\th$ respectively. 
These families are analytic in $\th, t$ when considered as operators satisfying 

\begin{equation} 
{U}^{s,c,u}_{\th}(t)X^{s,c,u}_{\th}= X^{s,c,u}_{\th+\omega t}. 
\end{equation} 

Let $\Pi_{\th}^{s,c,u}$ the projections associated to this
splitting. Assume furthermore there exist 
${\beta}_1, {\beta}_2 ,\,{\beta}^\pm_3>0$, 
$\alpha_1, \alpha_2 \in (0,1)$ and $ {C}_h>0$
independent of $\th \in \overline{D_\rho}$ 
satisfying 
${\beta}_3^+ <{\beta}_1 $,
$ {\beta}_3^- <{\beta}_2 $ 
and such that the splitting is characterized by the
following rate conditions:

\begin{equation}
\begin{array}{c}
\|{U}^s_{\th}(t) \|_{\rho,Y,X} \leq {C}_h \frac{e^{-{\beta}_1 t}}{t^{\alpha_1}},\qquad t > 0,\\
\|{U}^u_{\th}(t)  \|_{\rho,Y,X} \leq {C}_h \frac{e^{{\beta}_2 t}}{|t|^{\alpha_2}},
\qquad t < 0,\\
\|{U}^c_{\th}(t) \|_{\rho,X,X} \leq {C}_h e^{{\beta}^+_3 |t|},
\qquad t > 0 \\ 
\|{U}^c_{\th}(t) \|_{\rho,X,X} \leq {C}_h e^{{\beta}_3^- |t|},
\qquad t   < 0. 
\end{array} 
\end{equation}
Let $F^{s,u} \in \mathcal A_{\rho,Y}$ taking values in  $Y^s$ (resp. $Y^u$  ).
Consider the equations 
\begin{equation}\label{eqProj}
\partial_\omega \Delta^{u,s}(\th) -\mathcal L(\th)\Delta^{u,s}(\th) =F^{u,s}(\th)
\end{equation}

Then there are unique bounded solutions for 
\eqref{eqProj} which are given by the following formulas:
\begin{equation}\label{solucionProjs}
\Delta^{s}(\th)=\int_0^{ \infty}U^{s,u}_{\th-\omega
\tau}(\tau)  F^s(\th-\omega \tau)\,d\tau.
\end{equation}  
and
\begin{equation}\label{solucionProju}
\Delta^{u}(\th)=\int_{-\infty}^{0}U^{s,u}_{\th-\omega
\tau}(\tau) F^u(\th-\omega \tau)\,d\tau.
\end{equation}  
Furthermore, the following estimates hold
\begin{equation*}
\|\Delta^{s,u}\|_{\rho,X_{\th }^{s,u}} \leq  
C  \|\Pi^{s,u} _{\th  }\|_{\rho,Y,Y} \|F \|_{\rho,Y_{ \th }^s}. 
\end{equation*}    
\end{lemma}
\begin{remark}
The assumptions of the previous Lemma are very similar to the standard
setup of the theory of dichotomies, but we have to take care of the
fact that the evolution operators are smoothing and the perturbations
unbounded.
\end{remark}

\begin{proof}
The proof is based on the integration of the equation
along the characteristics by using $\th+ \omega t$. We give the proof for
the stable case, the unstable case being symmetric (for negative
times). Furthermore, the proof is similar to the one in
\cite{fontichdelLS071} up to some modifications of the functional
spaces. Denote $\tilde \Delta ^s(t)=\Delta^s(\th+\omega t)$.  By the variation of parameters  formula (Duhamel formula), which is valid in the 
mild solutions context (see \cite{pazy}) one has 
\begin{equation} \label{formulaldeltas}
\tilde \Delta^s(t)=U^s_{\th}(t)
\tilde \Delta^s(0)+\int_0^tU^{s}_{\th}(t-z)F^s(\th+\omega z)\,dz.
\end{equation}
Since the previous formula  is valid for all $\th\in D_\rho\supset \TT^\ell$ we can use it substituting $\th$ by 
$\th -\omega t$ and then
\begin{eqnarray*}
\Delta ^s(\th)=U^s_{\th-\omega t}(t)\Delta ^s(\th-\omega
t)+\int_0^tU^s _{\th-\omega (t-z)}(t-z) F^s(\th-\omega(t- z))\,dz .
\end{eqnarray*}
By the previous bounds on the semi-group, we have that $U_{\th-\omega t}(t) \Delta ^s(\th-\omega t)$
goes to 0 when $t$ goes to $\infty$.  

Hence this leads to the following representation formula after replacing $t-z$ by $\tau >0$ in the integral
\begin{equation}\label{solucioalest1}
\Delta^s(\th)=\int_0^{ \infty}U_{\th-\omega
\tau}(\tau)  F^s (\th-\omega \tau)\,d\tau.
\end{equation}  
Furthermore, from the previous formula, one has that $\Delta^s$ is analytic in $\th$.

We now estimate the integral in \eqref{solucioalest1} to show 
that it converges and to establish bounds on it. 
Notice that the operator $U^s_\th(t)$ maps $Y^s_\th $ into
$X^s_\th $ continuously and the following estimate holds for every
$\th \in D_\rho$ and every $t>0$ 
$$\|U^s_\th(t) F^s (\th)\|_{X^s_{ \th }} \leq
\frac{C}{t^{\alpha_1}} e^{-\beta_1 t } \| F^s (\th) \|_{Y^s_{ \th }}.$$ 

The exponential bound in {\bf SD3.2} ensures the convergence at infinity of the
integral and the fact that $\alpha_1 \in (0,1)$ ensures the convergence
at $0$ and one gets easily the desired bound. 

The unstable case can be obtained by reversing the direction of 
time or given a direct proof which is identical to the present one. 
\end{proof}


\section{Perturbation theory of hyperbolic bundles in an infinite-dimensional framework}\label{change-nondeg}

In this section we develop a perturbation theory for hyperbolic
bundles and their smoothing properties. 
 We consider a slightly more general framework than the one
introduced in the previous sections since 
we hope that the results in this section could be useful 
for other problems (e.g in dissipative PDE's). 
In particular, we note that we only assume that the
spaces $X$ and $Y$ are Banach spaces.
Also, we do not need to assume that the (unbounded)  vector field 
$\X$ giving the equation is Hamiltonian. In agreement with 
previous results, we note that we do not 
assume that the equations define an evolution for all initial 
conditions. We only assume that we can define evolutions in 
the future (or on the past)  of the linearization in some spaces. 
This is obvious for the linear operator and  in this section, 
we will show that this is persistent under small perturbations. 

The theory of perturbations of bundles for evolutions in 
infinite dimensional spaces has a long history.
 See for example~\cite{henry, PlissS99}. A treatment of 
partial differential equations has already been considered in 
the literature. For example in \cite{ChowL95, ChowL96}.

Our treatment has several important differences with the 
above mentioned works. Among them: 
1) We study the stability of smoothing properties, 2) We take advantage 
of the fact that the dynamics in the base is a rotation, so that we 
obtain results in the analytic category, which are false when the dynamics
in the base is more complicated. 3) We present our main results in 
an a-posteriori format, which, of course, implies the standard persistence 
results but has other applications such as validating numerical
or asymptotic results. 4) We present very quantitative estimates of the 
changes of the splitting and its merit figures under perturbations. This 
is needed for our applications since we use it as an ingredient of 
an iterative process and we need to show that it converges.

The main result in this Section is Lemma~\ref{iterNH} which shows 
that the invariant splittings and their smoothing properties when 
we change the linearized equation. Of course, in the applications in 
the iterative Nash-Moser method, the change of the equation will be 
induced by a change in the approximate solution.

\begin{lem}\label{iterNH}
Assume that $A(\th)$ is a family of linear maps as before. 
Let $\tilde A(\th)$ be another family 
such that $\|\tilde A-{A}\|_{\rho,X,Y}$ is small enough. 

Then there exists a family of analytic splittings
\begin{equation*}
X=\tilde X^s_{\th}\oplus
\tilde{X}^c_{\th }\oplus \tilde{X}^u_{\th}
\end{equation*} 
which is
invariant under the linearized equations
\begin{equation*}
\frac{d}{dt}\Delta =\tilde{A}(\th+\omega t)\Delta
\end{equation*} 
 in the sense that the following hold
\begin{equation*}
\tilde{U}^{s,c,u}_{\th}(t)\tilde{X}^{s,c,u}_{\th}=\tilde{X}^{s,c,u}_{\th+\omega t}. 
\end{equation*}
We denote $\tilde\Pi_{\th}^{s,c,u}$  the projections associated to this splitting. 
Then
there exist $\tilde{\beta}_1 ,\, \tilde{\beta}_2 ,\, \tilde{\beta}^+_3,\tilde{\beta}^-_3>0$, $\tilde \alpha_1, \tilde \alpha_2 \in (0,1)$ and $\tilde{C}_h>0$ independent of $\th$ satisfying $\tilde{\beta}_3<\tilde{\beta}_1  $, $\tilde{\beta}_3<\tilde{\beta}_2 $ and such that the splitting is characterized by the following rate conditions:
\begin{equation} 
\begin{split} 
& \|\tilde{U}^s_{\th}(t) \|_{\rho,Y,X} \leq \tilde{C}_h \frac{e^{-\tilde{\beta}_1 t}}{t^{\tilde \alpha_1}},\qquad t > 0,\\
&\|\tilde{U}^u_{\th}(t) \|_{\rho,Y,X} \leq \tilde{C}_h \frac{e^{\tilde{\beta}_2 t}}{|t|^{\tilde \alpha_2}},
\qquad t < 0,\\
&\|\tilde{U}^c_{\th}(t) \|_{\rho,X,X} \leq \tilde{C}_h e^{\tilde{\beta}^+_3 t},
\qquad t >0 \\
& \|\tilde{U}^c_{\th}(t) \|_{\rho,X,X} \leq \tilde{C}_h e^{\tilde{\beta}^-_3 |t|},
\qquad t <0. 
\end{split} 
\end{equation}
Furthermore the following estimates hold 
\begin{align}
\label{proj1NH} \|\tilde \Pi_{\th}^{s,c,u}-\Pi_{\th}^{s,c,u}\|_{\rho,Y,Y
  } &\leq C\|\tilde{A}-A\|_{\rho,X,Y},\\
\label{betas}
|\tilde{\beta}_i- \beta_i  | 
 &\leq C  \|\tilde{A}-A\|_{\rho,X,Y},\qquad i=1,2,3^\pm ,\\
\label{alphas}\tilde{\alpha}_i & = \alpha_i,\qquad i=1,2 \\
\tilde C_h & = C_h.
\end{align} 
\end{lem}

\begin{proof} 
We want to find invariant subpaces for the linearized evolution equation. 
We concentrate on the stable subspace, the theory for the other bundles being similar. 
We do so by finding a
family of linear maps indexed by $\th$, denoted
$\M_\th :X^s_\th \to X^{cu}_\th\equiv X^c_\th \oplus X^u_\th$ 
in such a way that the graph of $\M_\th$ is invariant under the equation. Note that since we do not assume that the equation defines a flow, the fact that we can evolve the elements  in the graph  in the 
future is an important part of the conclusions.  We will also show 
that the family of maps depends analytically in $\th$. 

{\bf Step 1: Construction of the invariant splitting. } 

We will consider first the case of the stable bundle. The others are 
done identically. We will first try to characterize the initial conditions
of the linearized evolution equation  that lead to 
a forward evolution which is a contraction. 
 We will see that these lie in a space.  We will formulate the 
new space as the graph of a linear function $\M_\th$ from $X^s_\th$
to $X^{cu}_\theta$. We will show that if such a characterization 
was possible, $\M_\th$ would have to satisfy some equations. 
To do that, we will formulate the problem of existence of 
forward solutions and the invariance of the bundle as two (coupled) 
fixed point problems (see \eqref{temp1} and 
\eqref{temp2}.)  One fixed point problem will formulate the 
invariance of the space 
, and the other fixed point problem the existence
of forward solutions.  We will show that, in some appropriate 
spaces, these two fixed point problems  can be studied using the 
contraction mapping principle. The definition of the spaces will 
be somewhat elaborate since they will also encode the analytic dependence
on the initial conditions, which is natural if 
we want to show the analytic dependence on $\th$ of 
the invariant spaces. 

Note that, since the main tool will be  a contraction 
argument, it follows that the main result is an a-posteriori 
result. Given approximate solutions of the invariance equations
(obtained e.g numerically or through formal expansions, etc. ) 
one can find a true solution close to the approximate one. 
We leave to the reader the recasting of Lemma~\ref{iterNH} in this style.

Now, we implement in detail the above strategy: We first derive 
the functional equations, then, specify the spaces. 

We start by considering the linearized equation 
with an initial phase $\th$. For subsequent analysis, it will be 
important to study the dependence on $\th$ of the solutions. 
Eventually, we will show that the new invariant spaces depend
analytically on $\th$.  This will translate in the geometric 
properties of the bundles. Consider

\begin{equation} \label{lin-eq-temp}
\frac{d }{dt}W_\th(t) = \tilde{A}(\th+\omega t)W_\th(t)
\end{equation} 
Note that we use the index $\th$ to indicate that we are considering the 
equation with initial phase $\th$. 

We write \eqref{lin-eq-temp} as
\begin{equation} \label{lin-eq-2}
\frac{d }{dt}W_\th(t) = {A}(\th+\omega t)W_\th(t) + B(\th+\omega t)W_\th(t)
\end{equation} 
with $B= \tilde A - A $. Denote $\gamma=\|\tilde A-A\|_{\rho,X,Y} \equiv
\|B\|_{\rho, X,Y}$, which we will assume to be small. 

We recall that this is an equation for $W_\th$ and 
that we are not assuming solutions to exist. Indeed, one of 
our goals is to work out conditions that ensure that forward
solutions exist.  Hence, we will manipulate the equation
\eqref{lin-eq-temp} to obtain some conditions.

We compute the evolution of the projections of $W_\th(t)$ 
along the invariant bundles by the linearized equation when 
$B \equiv 0$.  For $\sigma = s,c,u$  we have: 

\begin{equation}\label{lin-eq-3} 
\begin{split} 
\frac{d}{dt}\left( \Pi^\sigma_{\th + \omega t} W_\th(t) \right) 
=& \left( \omega \cdot \partial_\th \Pi^\sigma_{\th + \omega t} \right) W_\th(t) 
+ \Pi^\sigma_{\th + \omega t} \left(\frac{d}{dt} W_\th(t)\right) \\ 
=& \left( \omega \cdot \partial_\th \Pi^\sigma_{\th + \omega t} \right) W_\th(t) 
+ \Pi^\sigma_{\th + \omega t} A(\th + \omega t)  + 
\Pi^\sigma_{\th + \omega t} 
 B(\th + \omega t) W_\th(t) \\
=& A^\sigma(\th + \omega t) \Pi^\sigma_{\th + \omega t} W_\th(t) +
\Pi^\sigma_{\th + \omega t} B(\th + \omega t) W_\th(t) 
\end{split} 
\end{equation} 
In the last line of \eqref{lin-eq-3}, we have used that 
the calculation in the first two  lines of \eqref{lin-eq-3} is also 
valid when $B = 0$ and that, in that case, the invariance of 
the bundles under the $A$  evolution implies that all the terms 
appearing can be subsumed into $A^\sigma$ which only depends on 
the projection on the bundle. 

Of course the same calculation is valid for the projections over 
the center-unstable ( and center-stable, etc.) bundles. 
We denote by $\Pi^{cu}_\th = \Pi^c_\th + \Pi^u_\th$ the 
projection over the center-unstable bundle.  Note that 
$\Pi^s_\th + \Pi^{cu}_\th = \Id$. 

Our goal now is to try to find a subspace in which the 
solutions of \eqref{lin-eq-temp} (equivalently \eqref{lin-eq-3}) 
can be defined forward in time.

We will assume that this space where solutions can be defined 
is given as the graph of a linear function $\M_\th$ from 
$X^s_\th$ to $X^c_\th \oplus X^u_\th$. 
That is, we introduce the notation
$W_\th^{cu}(t) = \Pi^{cu}_{\th + \omega t} W_\th(t) $, 
$W_\th^{s}(t) = \Pi^{s}_{\th + \omega t} W_\th(t) $
and  we will assume that the solutions of 
\eqref{lin-eq-temp} have the form 
\[
W_\th^{cu}(t) =  \M_{\th + \omega t} W_\th^s(t) 
\]

 We will have to show 
that this linear subspace of $X$  can indeed be found and, show that it depends analytically 
on $\th$. For any $T> 0 $, if there were solutions of 
the equation satisfied by $W^{cu}_\th$ we would have 
Duhamel's formula. Then, imposing that it is in the graph:

\begin{equation}\label{temp1-a}\begin{split}
\M_{\th} W_\th^s(0) =& W_\th^{cu}(0) \\
= &
U^u_{\th +\omega T}(-T)\M_{\th +\omega T}W_\th^s(T)\\
&\phantom{AAAAAA}+\int_0^T U^u_{\th +\omega t}(t-T)\Pi^{cu}_{\th+\omega t} B(\th+\omega t)(Id+\M_{\th+\omega t})W_\th^s(t)\, dt. 
\end{split} 
\end{equation}

Similarly, one has 
\begin{equation}\label{temp2-a}
W_\th^s(t)=U^s_{\th +\omega t} W_\th^s(0)+\int_0^t U^s_{\th +\omega (t-\tau)}(t-\tau)\Pi^{cu}_{\th+\omega \tau} B(\th+\omega \tau)(Id+\M_{\th+\omega \tau})W_\th^s(\tau)\, dt. 
\end{equation}

Notice that the fact that  \eqref{temp2-a} is linear 
implies that if  its solutions  are unique, then 
 $W_\th^s(t)$ depends linearly 
on $W^s_\th(0)$ (it depends very nonlinearly on $\M_\th$). 
 We will write $W^s_\th(t) = \N_\th(t) W^s_\th(0)$ where 
$\N_\th(t)$ is a linear operator. 

 We have then
\begin{equation}\label{temp1}\begin{split}
\M_{\th}
= &
U^u_{\th +\omega T}(-T)\M_{\th +\omega T}\N_\th(T)\\
&\phantom{AAAAAA}+\int_0^T U^u_{\th +\omega t}(t-T)\Pi^{cu}_{\th+\omega t} B(\th+\omega t)(Id+\M_{\th+\omega t}) \N_\th(t) \, dt
\end{split} 
\end{equation}

Similarly, we have that \eqref{temp2-a} is implied 
by 
\begin{equation}\label{temp2} 
\N_\th^s(t)=U^s_{\th +\omega t}(0)+\int_0^t U^s_{\th +\omega (t-\tau)}(t-\tau)\Pi^{cu}_{\th+\omega \tau} B(\th+\omega \tau)(Id+\M_{\th+\omega \tau})\N_\th^s(\tau).
\end{equation}

We can think of \eqref{temp1} and \eqref{temp2} 
as equations for the two unknowns $\M$ and $\N_\th$ where 
$\M$ will be a function of $\th$ and $\N$ a function of 
$\th,t$. 

Note that \eqref{temp1} and \eqref{temp2} are 
already written as fixed point equations for
the operators defined by the right hand side of the equations. 
It seems intuitively clear that the R.H.S. of 
the equations will be contractions since the linear terms 
involve a factor $B$ which we are assuming is small. 
Of course, to make this intuition  precise, we  have 
to specify appropriate Banach spaces and carry out some 
estimates. After the spaces are defined, 
the estimates are somewhat 
standard and straightforward.  We point out that operators similar 
to \eqref{temp1} appear in the perturbation theory of 
hyperbolic bundles and operators similar to 
\eqref{temp2} appear in the theory of perturbations of 
semigroups.  The integral equations are also very common in 
the study of neutral delay equations. 

\subsubsection{Definition of spaces} 
Let $\rho >0$. For  $\th \in \overline{D_\rho}$ we denote by 
$\mathcal L(X^s_\th,X^{cu}_\th)$ the space  of bounded
linear maps from $X^s_{\th}$ into $X^{cu}_{\th}$.
We considered it endowed with the standard  supremum norm of linear operators. 

Denote also by
$\mathcal L_\rho(X^s,X^{cu})$ the space of analytic 
mappings  from $D_\rho$
into the space of linear operators in $X$ that to each 
$\th \in D_\rho$, assign
  a linear operator in 
$\mathcal L(X^s(\th),X^{cu}(\th))$. We also require from the maps 
in $\mathcal L_\rho(X^s,X^{cu})$  that they extend 
continuously to the boundary of $D_\rho$. We endow 
$\mathcal L_\rho(X^s,X^{cu})$ with the topology
of the supremum norm, which makes
 it into a Banach space. 

We also introduce  the standard $C^0( [0,T], \L_\rho(X^s, X^{cu}))$, endowed with the supremum norm.  
For each $\th \in D_\rho$ we denote 
$C^0_\th( [0,T], \L(X^s, X^c))$ the space of 
continuous functions which for every $t \in [0,T]$,
assign a linear operator in 
$\L( X^s_{\th + \omega t}, X^{cu}_{\th +\omega t})$. 
Of course, the space is endowed with the supremum norm. 
For typographical reasons, we will abreviate  the 
above spaces to $C^0$ and $C^0_\rho$. It is a standard  result  that the above spaces are Banach spaces 
when endowed with the above norms. 

\subsubsection{Some elementary estimates}
We denote by $\Tau_1$, $\Tau_2$ the operators given  
by the R.~H.~S.  of the equations \eqref{temp1} and 
\eqref{temp2} respectively. For typographical reasons, we just denote $\| B\| = \sup_{\th \in D_\rho} \| B(\th) \|_{X,Y}$. 

Using just the triangle inequality and bounds on the semi-group $U^s_\th$, 
we have:

\begin{eqnarray*} 
 \| \Tau_2(\M, \N) - \Tau_2(\tilde \M, \tilde \N)  \|_{C^0} 
 \le C \Big(   (1 + \| \M\|_{\L_\rho} \big)\| B\| ) \| \N -\tilde \N\|_{C^0}\\
+ \max( \| \N \|_{C^0}, \| \tilde \N\|_{C^0}) \Big ) \| \M - \tilde \M\|_{C^0}\\
 \| \Tau_1(\M, \N) - \Tau_1(\tilde \M, \tilde N)  \|_{C^0} 
\le\left(  C_h T^{-\alpha_1} e^{-\beta_1 T} \| \M \|_{\L_\rho} + 
C (1 + \| \M\|_{\L_\rho} \| B\|  \right) \| \N -\tilde \N\|_{C^0_\rho} \\
+ C_h T^{-\alpha_1} e^{-\beta_1 T} \| \tilde M  - \M \|_{\L_\rho} 
+ C \| B\| \max( \| \N \|_{C^0}, \| \tilde \N\|_{C^0}) \| \M - \tilde \M\|_{C^0}\\
\end{eqnarray*}

Since $\|U^s_\th\| \le A $, we choose $\S = \{(\N, \M ) \le 2 A\}$. 
We first fix $T$ large enough so that $C T^{-\alpha_1} T e^{-\beta_1 T} \le 10^{-2}$. 
Then, we see that if $\| B\| $ is small enough, $(\Tau_1, \Tau_2)(\S) \subset \S$. 

Furthermore, under another  smallness condition in $\| B \|$, using the previous bounds, we 
see that $(\Tau_1, \Tau_2)$ is a  contraction in $S$.

Therefore, with the above choices  we can get solutions of \eqref{temp1}, \eqref{temp2} 
which are sufficient conditions to obtain a forward evolution and 
that the graph is invariant under this evolution.

\subsubsection{Some small arguments to finish the construction of the invariant subspaces} 
Since we have the function $W$ defined in all $D_\rho$, it follows 
that the function $\W(t) = W(\th + \omega t )$ is defined for all
time as desired. 
The argument also shows that for a fixed $\th$, the function solves 
the linearized equation for a short time. 
Of course, the argument can be done in the same way for other dichotomies
running the time backwards. Hence we obtain the  stability of the splitings 
$X^{sc}$ and $X^u$. The space $X^c$ can be reconstructed as 
$X^c = X^{cu}\cap X^{sc}$.

{\bf Step 2. Estimates on the projections.} 
To get the bounds for the projections we 
use  the same argument as in \cite{fontichdelLS071}. 
We only give the argument for the stable subspace.
Let $\M^{cu}_\th $ be the linear map whose graph gives 
$\tilde X^{cu}_{\th} $.

We write 
\begin{align*}
\Pi^s_{\th } \xi = (\xi^s,0), & 
\qquad \tilde\Pi^s_{  \th}\xi =  (\tilde \xi^s,\M^s_\th \tilde \xi^s), \\
\Pi^{cu}_{\th} \xi = (0,\xi^{cu} ), & 
\qquad \tilde \Pi^{cu}_{ \th}\xi =  (\M^{cu}_\th \tilde \xi^{cu}, \tilde \xi^{cu}) ,
\end{align*}
and then 
\begin{align*}
\xi^s &=\tilde \xi^s +
\M^{cu}_\th \tilde \xi^{cu}, \\
\xi^{cu} &= \M^s_\th \tilde \xi^s +  \tilde \xi^{cu} .
\end{align*}
Since $\M^s_\th$ and $\M^{cu}_\th$ are $O(\gamma)$ in $\L(X,X)$  we can write
$$
\left( \begin{array}{c}
\tilde \xi^s \\
\tilde \xi^{cu}
\end{array}\right) 
= 
\left( \begin{array}{cc}
\Id & \M^{cu}_\th  \\
 \M^s_\th & \Id
\end{array}\right) ^{-1}
\left( \begin{array}{c}
\xi^s \\
\xi^{cu}
\end{array}\right) 
$$
and then 
deduce that 
$$
\|(\tilde\Pi^s_{  \th  } - \Pi^s_{  \th  } )\xi \|_Y
\le \|(\tilde \xi ^s -  \xi^{s}, \M^s_\th \tilde \xi^{s}) \|_Y
\le  C\gamma .
$$

{\bf Step 3. Stability of the  smoothing properties.}

In this step, we will show that the smoothing properties of 
the cocycles are preserved under the lower order  perturbations considered 
before. That is, we will show that if we define the evolutions in the
invariant spaces constructed in Step 1 above, they satisfy bounds of the
form  in {\bf SD.3} but with slightly worse parameters. To be able to
apply this repeatedly, it will be important for us to develop estimates on
the change of the constants as a function of the correction. 

We will first study the stable case.  The unstable case is studied 
in the same way, just reversing the direction of time. The maps $U^s_\th$ and $\tilde U^s_\th$ satisfy the variational equations
$$
\frac{dU^s_\th}{dt}=A(\th+\omega t)U_\th^s(t)
$$
and 
$$
\frac{d\tilde U^s_\th}{dt}=\tilde A(\th+\omega t)\tilde U_\th^s(t). 
$$
Since $(U^s_\th-\tilde U^s_\th)(0)=0$, one has by the variation of parameters formula 
\begin{equation}\label{varparameters} 
\tilde U^s_\th (t)= U_\th^s (t) 
+ \int_0^t  U^s_\th (t-\tau)(\tilde A- A) (\th+\omega \tau)
 \tilde U_\th^s (\tau)\,d\tau,
\end{equation} 
for $t \ge 0$.

Let $\mathcal C_{\alpha,\beta,\rho}(X)$ be the space of continuous functions from $(0,\infty)$ into the space $\A_{\rho,\L(X,X)}$ endowed with the norm  
\[ 
|||U |||_{\alpha,\beta,\rho} = \sup_{\substack{\th \in D_\rho\\ t > 0} } 
||U(\th(t)) ||_{Y, X} e^\beta t^{\alpha} 
\]
We fix $\tilde A$, $A$ and $U_\th^s$ and consider the left hand-side of \eqref{varparameters} as an operator on $\tilde U_\th^s$, i.e. denote 
$$\mathcal T \bar U^s_\th (t)=U_\th^s (t) 
+ \int_0^t  U^s_\th (t-\tau)(\tilde A- A) (\th+\omega \tau)
 \tilde U_\th^s (\tau)\,d\tau. $$
Hence  
\eqref{varparameters} is just a fixed point equation. We note that the operator 
$\Tau$ is affine in its argument. We write it as 
$\Tau(U_\th^s)  = \mathcal O + \L(U_\th^s)$ where $\mathcal O$ is a constant vector and 
$\L$ is a linear operator. To show that $\Tau$ is a contraction, it suffices to 
estimate the norm of $\L$. We have 
$$
|||\L U_1-\L U_2 |||_{\alpha,\beta,\rho}  \leq C \gamma \Big ( t^{\alpha }e^{\beta t} \int_0^t \frac{e^{-\beta_1(t-\tau)}}{(t-\tau)^{\alpha_1}} e^{-\beta \tau}\tau^{-\alpha} d\tau \Big ) |||U_1-U_2 |||_{\alpha,\beta,\rho} .  
$$
We now estimate $$ C(t)=t^{\alpha }e^{\beta t} \int_0^t \frac{e^{-\beta_1(t-\tau)}}{(t-\tau)^{\alpha_1}} e^{-\beta \tau}\tau^{-\alpha} d\tau . $$ We have 
$$
C(t)=t^{\alpha } \int_0^t \frac{e^{(\beta-\beta_1)(t-\tau)}}{(t-\tau)^{\alpha_1}}\tau^{-\alpha} d\tau .
$$
 Changing variables, one gets 
$$
C(t)=t^{\alpha } \int_0^t \frac{e^{(\beta-\beta_1)z}}{(t-z)^{\alpha}}z^{-\alpha_1} dz .
$$

We now choose $\beta$ such that $\beta <\beta_1$ denoting $\beta=\beta_1-\ep$. Making the change of variables $z=tu$ in the integral, one gets 
$$
C(t)=t^{1-\alpha_1 } \int_0^1 \frac{e^{-\ep t u}}{(1-u)^{\alpha}}u^{-\alpha_1} du .
$$

This is clearly bounded for $t \leq 1$ since $\alpha \in (0,1)$ and $1-\alpha_1 >0$. We now consider the case $t >1$. There exists a constant $C>$ universal such that the following estimate holds 
$$
e^{-t\ep u}\leq \frac{C}{(1+t\ep u)^{1-\alpha_1}}
$$
for any $t,u \geq 0$. Therefore we estimate for $t>1$
$$
C(t)\leq Ct^{1-\alpha_1 } \int_0^1 \frac{du}{(1-u)^{\alpha}u^{\alpha_1}(1+\ep t u )^{1-\alpha_1}},
$$
which is uniformly bounded as $t$ goes to $\infty$. Recalling that $||| \mathcal LU_1-\mathcal L U_2 |||_{\rho, \alpha_1,\beta_1} \leq C \gamma$
where $C$ is the constant we just computed, we obtain that
$\mathcal L$ is a
contraction in the space $\mathcal C_{\alpha_1,\beta,\rho}(X)$ for any $\beta<\beta_1$ and any $\alpha_1 \in (0,1)$ when $\gamma$ is sufficiently small. 
\end{proof}

The first consequence of Proposition \ref{iterNH}  is that in the iterative step the small
change of $K$ produces a small change in the invariant splitting and in the hyperbolicity constants.
\begin{corollary}\label{cor1:iterNH}

Assume that $K $ satisfies the hyperbolic non-degeneracy Condition
\ref{ND1} and that 
$\|K-\tilde{K}\|_{\rho,X}$ is small enough. If we denote $\tilde{A}(\th) = D\X(K)$, there exists an analytic family of  splitting for $\tilde{K}$, i.e. 
\begin{equation*}
X=X^s_{{\tilde{K}(\th)}}\oplus
X^c_{{\tilde{K}(\th)}}\oplus X^u_{{\tilde{K}(\th)}}
\end{equation*} 
which is
invariant under the linearized equation \eqref{varNH} (replacing $K$ by $\tilde{K}$) in the sense that 
\begin{equation*}
\tilde{U}^{\sigma}_{\th}(t)X^{\sigma}_{\tilde{K}(\th)}=X^{\sigma}_{\tilde{K}(\th+\omega t)}.  \qquad \sigma = s,c,u
\end{equation*}
We denote $\Pi_{\tilde{K}(\th)}^s$, $\Pi_{\tilde{K}(\th)}^c$ and
$\Pi_{\tilde{K}(\th)}^u$ the projections associated to this splitting. There exist $\tilde{\beta}_1 ,\, \tilde{\beta}_2 ,\, \tilde{\beta}^+_3,\tilde{\beta}^-_3>0$, $\tilde \alpha_1, \tilde \alpha_2 \in (0,1)$ and $\tilde{C}_h>0$ independent of $\th$ satisfying $\tilde{\beta}^+_3<\tilde{\beta}_1  $, $\tilde{\beta}^-_3<\tilde{\beta}_2 $ and such that the splitting is characterized by the following rate conditions:
\begin{eqnarray*}
\|\tilde{U}^s_{\th}(t) \|_{\rho,Y,X} \leq \tilde{C}_h \frac{e^{-\tilde{\beta}_1 t}}{t^{\tilde \alpha_1}},\qquad t > 0,\\
\|\tilde{U}^u_{\th}(t) \|_{\rho,Y,X} \leq \tilde{C}_h \frac{e^{\tilde{\beta}_2 t}}{|t|^{\tilde \alpha_2}},
\qquad t < 0,\\
\|\tilde{U}^c_{\th}(t) \|_{\rho,X,X} \leq \tilde{C}_h e^{\tilde{\beta}^+_3 t},
\qquad t >0 \\
\|\tilde{U}^c_{\th}(t) \|_{\rho,X,X} \leq \tilde{C}_h e^{\tilde{\beta}^-_3 |t|},
\qquad t <0. 
\end{eqnarray*}  
Furthermore the following estimates hold 
\begin{align}
\label{cor1:proj1NH} \|\Pi_{{\tilde{K}(\th)}}^{s,c,u}-\Pi_{{K(\th)}}^{s,c,u}\|_{\rho,Y,Y
  } &\leq C\|\tilde K-{K}\|_{\rho,X},\\
\label{cor1:betas}|\tilde\beta_i - {\beta}_i|  &\leq C  \|\tilde K-{K}\|_{\rho,X},\qquad i=1,2,3 ,\\
\label{cor1:alphas}|\tilde\alpha_i - {\alpha}_i|  &\leq C  \|\tilde K-{K}\|_{\rho,X},\qquad i=1,2\\
\tilde C_h & = C_h.
\end{align} 
\end{corollary} 
\begin{proof} 
We just take $A(\th) = D\X(K(\th)) $, $\tilde A (\th)=  D\X(\tilde K(\th)) $,
$X^{s,c,u}_{ K(\th)} = X^{s,c,u}_{\th} $,
$X^{s,c,u}_{\tilde K(\th)} = \tilde X^{s,c,u}_{\th} $, 
$\Pi^{s,c,u}_{ K(\th)} = \Pi^{s,c,u}_{\th} $ and
$\Pi^{s,c,u}_{\tilde K(\th)} = \tilde \Pi^{s,c,u}_{\th} $
in Lemma~\ref{iterNH}
and we use that 
$\|\tilde A(\th)-A(\th) \|_{\rho,X,Y} \le  \|\X\|_{C^1} \|\tilde K(\th)- K(\th)\|_{\rho,X}$.
\end{proof}


\section{Solution of the cohomology equation on the center subspace}
\label{sol-center}

We now come to the solution of the projected equation \eqref{lin-eq}
 on the center
subspace. The first point which has to be noticed is that by the
spectral non-degeneracy assumption \ref{ND1} the center subspace $X^c_\th$
 is finite-dimensional (with dimension $2 \ell$).
 As a consequence, we end up with standard small divisors equations. 
This is in contrast with other studies of Hamiltonian 
partial differential equations like the Schr\"odinger equation for which 
there is an infinite number of eigenvalues on the 
imaginary axis (see \cite{bourgain} ) and the KAM theory is more involved. 
Another aspect of Definition \ref{ND1} is that the formal symplectic structure
 on $X$ restricts to a standard one on the center bundle. 
 Finally, it has to be noticed that by the finite-dimensionality assumption,
 all the issues related to unbounded operators become irelevant.

We denote 
\begin{equation*}
\Delta^c(\th)=\Pi^c_{\th} \Delta K (\th). 
\end{equation*}
The projected linearized equation \eqref{lin-eq} becomes
\begin{equation}\label{projCen}
\partial_{\omega}
\Delta^c (\th)-(D\X )\circ K\Delta^c(\th)=-\Pi^c_{\th} E(\th)=-E^c(\th).  
\end{equation}

We first recall a well-known result by R\"ussmann (see
\cite{Russmann76,Russmann76a,Russmann75,Llave01c}) which allows to
solve small divisor equations along characteristics. 
\begin{pro}\label{russVF}
Assume that $\omega\in D_h(\kappa,\nu)$ with 
$\kappa>0$ and $\nu \ge \ell-1$ and that $\M$ is a finite dimensional space. Let $h:D_\rho \supset\TT^\ell \rightarrow \M$ be a real analytic function with zero
average with values in $\M$. Then, for any  $0<\delta <\rho$ there exists a unique analytic solution 
$v:D_{\rho-\delta} \supset\TT^\ell \rightarrow \M$ 
of the linear equation 
\begin{equation*}
\sum_{j=1}^l \omega_j \frac{\partial v}{ \partial \th_j}=h
\end{equation*} 
having zero
average.
Moreover, if $h\in \A_{\rho,\M}$ then $v$ satisfies the
following estimate 
\begin{equation*}
\|v\|_{\rho-\delta,\M} \leq C \kappa\delta^{-\nu} \|h\|_{\rho, \M}, \qquad 0<\delta <\rho.
\end{equation*}
The constant $C$ depends on $\nu$ and the dimension of the torus $\ell$.
\end{pro}

As in \cite{fontichdelLS071} and  \cite{LGJV05}, we will find an explicit change of variables so that 
 the vector-field $D\X \circ K\Delta^c(\th)$ becomes a constant coefficient vector-field. Then we will be able to apply the small divisor result as stated in Proposition \ref{russVF} to the cohomology equations \eqref{projCen}. 

\subsection{Geometry of the invariant tori}

As it is well known in KAM theory, in a finite dimensional framework, maximal invariant tori are Lagrangian submanifolds and whiskered tori are isotropic. In our context of an infinite dimensional phase space $X$, the picture is less clear, but nevertheless, thanks to our assumptions (which are satisfied in some models under consideration), one can produce a non trivial solution. 

We prove the following lemma on the isotropic character of approximate invariant tori. 
\begin{lemma}\label{isotapprox}
Let $K: D_\rho\supset\mathbb{T}^\ell \rightarrow \mathcal{M}$, $\rho>0$, be a real analytic mapping.
Define the error in the invarianne equation as
\begin{equation*}
E(\th):= \partial_{\omega} K(\th)-\X(K(\th)).
\end{equation*} 
Let $L(\th)=DK(\th)^\perp J_cDK(\th)$ be the matrix which expresses 
the form  $K^*\Omega$
on the torus in the canonical basis. 

There exists
  a constant $C$ depending on $l$, $\nu$ and $\|DK\|_{\rho}$ such that 
\begin{equation*}
\|L\|_{\rho -2\delta,X_\th^c,X_\th^c}   \leq C \kappa \delta^{-(\nu+1)} \|E\|_{\rho,Y}, \qquad 0<\delta <\rho/2. 
\end{equation*}
In particular, if $E=0$  then $$L \equiv 0$$
\end{lemma}

\begin{proof}
By assumption $\bf H3.2$ we have that there exists a one-form $\alpha_K$ on the torus $\TT^\ell$ such that 
$$
K^*\Omega=d\alpha_K. 
$$
In coordinates on $\TT^\ell$, $\alpha_K$ writes 
$$
\alpha_K=g_K(\th)d\th. 
$$
Hence one has $L(\th)=Dg_K^\perp(\th)-Dg_K(\th)$ and the lemma follows from Cauchy estimates and Proposition \ref{russVF} (see also \cite{LGJV05}).
\end{proof}

\subsection{Basis of the center subspace $X^c_\th$}

We introduce a suitable representation of the center subspace $X^c_\th$. 
In \cite{LGJV05,fontichdelLS071,fontichdelLS07a} it is shown that  the change of variables
given by the following matrix
\begin{equation}\label{change}
[DK(\th), J_c^{-1}DK(\th)N(\th)].
\end{equation}
allows to transform the linearized equations in the 
center subspace into two cohomology equations  with constant coefficients. 

The argument presented in the references above 
works word by word here thanks to the fact that the  center subpace $X^c_\th$
is finite dimensional.  We will go over the main points in 
Section~\ref{normalizationprocedure}. We will start by recalling some 
symplectic properties. 

\subsubsection{Some symplectic preliminaries}
\label{symplecticproperties}

We prove the following lemma.

\begin{lemma}\label{omegadeg} 
The $2-$form $\Omega$ which is the restriction 
to the center subspace  is non-degenerate
 in the sense that $\Omega(u,  v ) = 0 \, 
\forall u \in X$ implies that $v = 0 $. 
\end{lemma}

\begin{proof}
A quick proof would follow from the fact that the symplectic form 
is non-degenerate at the origin. Then, because the non-degeneracy assumptions 
are open, it follows in a small neighborhood. The following argument gives a more global argument valid in all the center
manifold. 

By the non-degeneracy assumptions \ref{ND1}, there exist maps $U_\th^{s,c,u}(t)$ generating the linearizations on $X^{s,c,u}_\th$. These maps preserves $\Omega$. Indeed, one has: let $u(t),v(t)$ satisfy 
$$\frac{du(t)}{dt}=A(\th+\omega t)u(t)$$
and 
$$\frac{dv(t)}{dt}=A(\th +\omega t)v(t)$$
where $A(\theta)=J^{-1}\nabla^2 H\circ K(\theta).$ Then 
$$\Omega(u(t),v(t))=\Omega(u(0),v(0)). $$
Indeed, 
$$\frac{d}{dt}\Omega(u(t),v(t))=\Omega(\dot u(t),v(t))+\Omega(u(t),\dot v(t))$$
$$=<J^{-1}\nabla ^2 H\circ K(\th +\omega t) u(t), Jv(t)>+< u(t), J J^{-1}\nabla ^2 H\circ K(\th +\omega t) v(t)>$$
$$=-<\nabla ^2 H\circ K(\th +\omega t) u(t), v(t)>+<u(t), \nabla ^2 H\circ K(\th +\omega t)  v(t)>$$
since $\nabla^2 H\circ K$ is symmetric. Hence the result.

Therefore, we have for any $u,v \in X^{s,c,u}_\th$

$$
\Omega(u,v)=\Omega(U^{s,c,u}_\th (t)u, U^{s,c,u}_\th (t) v), \qquad t \in \RR^+, \RR, \RR^-.   
$$
Using now the estimates in \ref{ND1}, we have the following: the form $\Omega$ satisfies $\Omega(u,v)=0$ in the following cases
\begin{itemize}
\item $u,v\in X^s_{\th} $,
\item $u,v\in X^u_{\th} $,
\item $u\in X^s_{\th}\cup X^u_{\th}  $ and $v\in X^c_{\th} $,
\item $v\in X^c_{\th} $ and $v\in X^s_{\th}\cup X^u_{\th}  $
\end{itemize}

This implies that the form $\Omega$ restricted to the center bundle $X^c_{\th}$ is non degenerate and the lemma is proved. 

\end{proof}

The form $\Omega$ is then a symplectic form since we assumed that the restriction of the form to $X^c_{\th}$ is closed.  Denote by $J_c$ the restriction of the operator $J$ on $X^c_{\th}$. 
Finally we define the operator $M(\th)$ from $\RR^\ell$ into $X^c_\th$.  
\begin{equation} \label{definicioMtilde}
M(\th)=[DK(\th), J_c^{-1} DK(\th)N(\th)].
\end{equation}
Notice that by assumption $X^{c}_\th$ is isomorphic to $Y^c_\th$. We emphasize on the fact that the operator $M(\th)$ belongs to $X^c_\th$. Indeed, it is clear from the equation that $DK$ (by just differentiating) belongs to the center space and so is $J_c^{-1} DK(\th)N(\th)$ by the fact that we consider the restriction $J_c$ of $J$ to the center.

\subsection{Normalization procedure}
\label{normalizationprocedure}

Let $W: D_\rho \supset \TT^\ell \rightarrow X^c_\th$ be such that  

$$\Delta^c(\th)=M(\th)W(\th)$$
From now on, the proof is very similar to the one in \cite{LGJV05} and we
just sketch the proofs. We refer the reader to  \cite{LGJV05} for the details. 
The following first lemma provides a reducibility argument for {\sl exact} solutions of \eqref{eq:embedding}. We note that since  the space $X^c_\th$ is finite dimensional the symplectic form 
needs to be defined only in a very weak sense.  
\begin{lemma}\label{normalization}
Let $K$ be a solution of 
\begin{equation*}
\partial_{\omega} K(\th)=\X(K(\th))
\end{equation*} 
with $M$ be defined as above and $K(\TT^\ell)$ is an isotropic manifold. 
Then there exists an $\ell \times \ell$-matrix $S(\th)$ such that 
\begin{equation}\label{parbloc}
\partial_{\omega}M(\th)-A(\th) M(\th)=M(\th)\begin{pmatrix} 0_\ell & S(\th)\\ 0_\ell &
  0_\ell \end{pmatrix},
\end{equation}
where 
\begin{equation*}
S(\th)=N(\th) DK(\th)^\top [J_c^{-1}\partial_\omega(DKN)-A(\th) J_c^{-1}DKN](\th)
\end{equation*}
where we have denoted $A(\th)=J_c^{-1}D(\nabla H (K))$.  
\end{lemma}
\begin{proof}
By differentiating the equation, we clearly have that the first $\ell$ columns of the matrix 
$$W(\th)=A(\th)M(\th)-\partial_\omega M(\th)$$
are zero. Now write 
$$W_1(\th)=A(\th)J_c^{-1} DK(\th)N(\th)-J_c^{-1} \partial_\omega(DK(\th)N(\th)). $$
Easy computations show that 
$$W_1(\th)=A(\th)J_c^{-1} DK(\th)N(\th)-J_c^{-1} \partial_\omega(DK(\th))N(\th)$$
$$+J_c^{-1} DK(\th)N(\th)\partial_\omega(DK^\top(\th))N(\th)+J_c^{-1} DK(\th)N(\th)DK(\th)^\top \partial_\omega(DK(\th))N(\th).$$
But since $DK$ and $J_c^{-1} DK(\th)N(\th)$ form a basis of the center subspace, one can write 
$$W_1=DK\,S +J_c^{-1} DKN\,T. $$

We will prove that $T=0$, giving the form of the matrix in the lemma. Multiply the previous equation by $DK(\th)^\top J_c$; then by the lagrangian character of $K$, we have 
$$DK(\th)^\top J_c W_1(\th)=T. $$
Hence using straightforward computations, we have that the second term plus the fourth term in $DK(\th)^\top J_c W_1(\th)$ is zero and the first term plus the third term in $DK(\th)^\top J_c W_1(\th)$ is equal to  
$$\big ( DK^\top D(\nabla H(K))J_c^{-1}+\partial_\omega (DK)^\top \big )DKN. $$ 
But using the fact the symplectic form is skew-symmetric, the quantity into parenthesis is just the derivative of the equation. Hence it has to be zero. 

We now check the expression of the matrix $S$.  We multiply by $N DK^\top$ to have 
$$S=N DK^\top W_1=N DK^\top \big ( A(\th)J_c^{-1} DK(\th)N(\th)-J_c^{-1} \partial_\omega(DK(\th)N(\th))\big ). $$
This gives the result. 
\end{proof}

The next lemma provides a generalized inverse for the operator $M$.

\begin{lemma} \label{isinvariant} 
Let $K$ be a solution of \eqref{eq:embedding}. Then the matrix 
${M}^\perp J_c {M} $ is invertible
and
\begin{equation*}
({M}^\perp J_c {M})^{-1}
=\begin{pmatrix} N^{\top}DK^\top J_c^{-1} DK \,N &
  -\Id_\ell \\ \Id_\ell & 0 
\end{pmatrix}.
\end{equation*}
\end{lemma}

We now establish a similar result for approximate solutions, i.e. solutions of \eqref{eq:embedding} up to error $E(\th)=\mathcal{F}_{\omega}(K)(\th)$. When $K$ is just an approximate solution, 
we define 
\begin{equation}\label{approxNorm}
(e_1, e_2) = \partial_{\omega}M(\th)- A( \th){M}(\th)-M(\th)
\begin{pmatrix} 
0_\ell & S(\th)\\ 
0_\ell & 0_\ell 
\end{pmatrix} .
\end{equation}
Using 
that $\partial_{\omega} DK(\th)- A(\th) DK(\th) = DE(\th)$
and the definition of $S$ above mentioned 
give $e_1= DE$ and $e_2 = O(\|E \|_{\rho,Y}, \|DE\|_{\rho,Y})$.

 We then get 
\begin{equation}\label{tempcVFNH}
[\partial_{\omega}M(\th)-A(\th)M(\th)]\xi(\th)+{M}(\th)\partial_{\omega}
\xi(\th)=-E^c(\th),
\end{equation}
 For the {\sl approximate} solutions of \eqref{eq:embedding}, we have the following
lemma.

\begin{lemma}\label{represVF}
Assume $\omega$ is Diophantine in the sense of definition
\ref{rotVF} and $\|E^c\|_{\rho,Y^c_{\th}}$ small enough. Then there exist a matrix
$B(\th)$ and vectors $p_1$ and $p_2$ such that, by the change of variables $\Delta^c=M
\xi$, the projected equation on the center subspace can be written 
\begin{eqnarray}\label{eqSDVF}
\left[ \begin{pmatrix} 0_l & S(\th)\\ 0_l & 0_l
\end{pmatrix}+B(\th) \right] \xi(\th)+\partial_{\omega}\xi(\th)=p_1(\th)+p_2(\th).
\end{eqnarray}
The following estimates hold
\begin{equation}
\|p_1\|_{\rho,X^c_{\th}} \leq C  \|E^c\| _{\rho,Y^c_{\th}},
 \end{equation}
\begin{equation}\label{estimp2VF}
\|p_2\|_{\rho-\delta,X^c_{\th}} \leq C \kappa \delta^{-(\nu+1)} \|E^c\|^2_{\rho,Y^c_{\th}}
 \end{equation}
and
\begin{equation}\label{estimbVF}
\|B\|_{\rho-2\delta,X^c_{\th}} \leq C \kappa \delta^{-(\nu+1)} \|E^c\|_{\rho,Y^c_{\th}},
 \end{equation}
where $C$ depends $l$, $\nu$, $\rho$, $\|N\|_{\rho}$, 
$\|DK\|_{\rho,Y}$, $|H|_{C^2(B_r)}$. Furthermore the vector $p_1$ has the expression 
$$
p_1(\th)=
\begin{pmatrix}
-N(\th)^\top  DK(\th)^\top E^c(\th) \\
DK(\th)^\top J_c E^c(\th)
\end{pmatrix}
$$
\end{lemma}

\begin{proof}
The proof follows more or less the one in \cite{LGJV05} with suitable adaptations due to our infinite dimensional setting. Notice however that the center subspace is finite dimensional. From the previous computations one has 
$$(e_1, e_2) = \partial_{\omega}M(\th)- A( \th){M}(\th)-M(\th)
\begin{pmatrix} 
0_\ell & S(\th)\\ 
0_\ell & 0_\ell 
\end{pmatrix} .
$$
Hence we have 
$$
M^\perp J_c \Big [ \partial_\omega M(\th)-A(\th)M(\th)\Big ]\xi(\th)=(M^\perp J_cM)\partial_\omega \xi=M^\perp J_cE_c. 
$$
Hence by the previous Lemma, 
\begin{eqnarray}
\left[ \begin{pmatrix} 0_l & S(\th)\\ 0_l & 0_l
\end{pmatrix}+(M^\perp J_c M)^{-1}(e_1,e_2) \right] \xi(\th)+\partial_{\omega}\xi(\th)=(M^\perp J_c M)^{-1}M^\perp J_cE_c. 
\end{eqnarray}
Hence denoting
$$
B(\th)=(M^\perp J_c M)^{-1}(e_1,e_2).
$$
Then direct computations give $p_1$ and $p_2$ and the desired estimates. 

\end{proof}

\subsection{Solutions to the reduced equations}\label{reducedEq}
We anticipate that from Lemma \ref{represVF}, the terms $B\xi$ and $p_2$ are quadratic in the error. Hence an approximate solution has
the form $\xi=(\xi_1,\xi_2)$ and solves
\begin{align}\label{approxSDVF}
S(\th)\xi_2(\th)-\partial_\omega
\xi_1(\th)&=-N(\th)^\top  DK(\th)^\top E^c(\th),\\
\partial_\omega \xi_2(\th)&=DK(\th)^\top J_c E^c(\th). \nonumber
\end{align}

We prove the following result, providing a solution to equations
\eqref{approxSDVF}.

\begin{pro}
There exists a solution $(\xi_1,\xi_2)$ of \eqref{approxSDVF} with the following estimates 
\begin{eqnarray*}
\|\xi_1\|_{\rho-\delta,X^c_{\th}} \leq C_1 \kappa \delta^{-\nu} \|E^c\|_{\rho,X^c_{\th }},\\ \qquad 
\|\xi_2\|_{\rho-2\delta,X^c_{\th }} \leq C_2 \kappa \delta^{-2\nu} \|E^c\|_{\rho,X^c_{\th }},
\end{eqnarray*}
for any $\rho \in (0,\delta/2)$ and where the constants $C_1,C_2$ just depend on $l$, $\nu$, $\rho$, $\|N\|_{\rho}$, 
$\|DK\|_{\rho,X^c_{\th}}$, $|\avg(S)|^{-1} $. 
\end{pro} 

\begin{proof}

In order to apply Prop. \ref{russVF}, one needs to study the average on the torus $\TT^\ell$ of $DK(\th)^\top J_c E^c(\th)$. To do so, we first consider assumption ${\bf H3.1}$ which gives in coordinates 
$$
DK^\top J_c DK=Dg^\top-Dg
$$
for some function $g$ on $\TT^\ell$. Now taking the inner product with $\omega$ and using the equation, one has 
$$
DK^\top J_c(E+\X(K))=Dg^\top\cdot \omega-Dg\cdot \omega .
$$
Therefore, the average of $DK^\top J_cE$ is the sum of the average of $Dg^\top \cdot \omega -Dg\cdot \omega$ which is zero and the average of $DK^\top J_c\X(K)$. Now notice that 
$$
DK^\top J_c\X(K)=i_{\X\circ K}K^* \Omega(.).
$$
Hence its average is zero by assumption {\bf H4}. As a consequence the average on $\TT^\ell$ of the R.H.S. $DK(\th)^\top J_c E^c(\th)$ is zero.Hence an application of Prop. \ref{russVF} gives the solvability in $\xi_2$ with the desired bound. Since the average of $\xi_2$ is free, one uses it and the twist condition to solve in $\xi_1$. This gives the desired result (see \cite{LGJV05} for details). 

\end{proof}

\section{Uniqueness statement}\label{sec:uniqueness}

In this section, we prove the uniqueness part of Theorem \ref{existence}. 

We assume that the embeddings $K_1 $ and $K_2 $
satisfy the hypotheses in Theorem \ref{existence}, in particular $K_1$ and $K_2$ are solutions
of \eqref{eq:embedding}.  
If $\tau\ne 0$ we write $K_1$ for $K_1\circ T_\tau$ 
which is also a solution.
Therefore
$\mathcal{F}_{\omega}(K_1)=\mathcal{F}_{\omega}(K_2)=0$.
By Taylor's theorem we can write 
\begin{equation}\label{unique}
\begin{split}
0=\mathcal{F}_{\omega}(K_1)-\mathcal{F}_{\omega}(K_2)=&D_{K}\mathcal{F}_{\omega}(K_2)(K_1-K_2)\\
&+\mathcal{R}(K_1,K_2),
\end{split}
\end{equation} 
where 
$$\mathcal R(K_1,K_2)=\frac{1}{2}\int_0^1 D^2\mathcal F_\omega(K_2+t(K_1-K_2))(K_1-K_2)^2\,dt.$$
Then, there exists $C>0$ such that 
\begin{equation*}
\|\mathcal{R}(K_1,K_2)\|_{\rho,Y} \leq C \|K_1-K_2\|_{\rho,X}^2. 
\end{equation*}
Hence we end up with the following linearized
equation 
\begin{equation} \label{linearized-5}
D_{K}\mathcal{F}_{\omega}(K_2)(K_1-K_2)=-\mathcal{R}(K_1,K_2).
\end{equation}
We denote $\Delta=K_1-K_2.$ Projecting \eqref{linearized-5} on the center subspace with $\Pi^c_{K_2(\th+\omega t)}$, writing 
$\Delta^c(\th) = \Pi^c_{K_2(\th)}\Delta (\th)  $ and
making the change of function $\Delta^c(\th)= M(\th) W(\th) $, 
where $ M$ is defined in \eqref{definicioMtilde} with $K=K_2$. 
We now perform the same type of normalization as in Section \ref{sol-center} to arrive to two small divisor equations of the type 
\begin{align}\label{approxUnique}
S(\th)\xi_2(\th)-\partial_\omega
\xi_1(\th)&=-N(\th)^\top  DK(\th)^\perp \mathcal{R}(0,0,K_1,K_2)(\th)^c,\\
\partial_\omega \xi_2(\th)&=DK(\th)^\top J_c \mathcal{R}(0,0,K_1,K_2)(\th)^c. \nonumber
\end{align}

We begin by looking for $\xi_2$. We search it 
in the form $\xi_2= \xi_2^\bot + \avg(\xi_2) $. We have
$\|\xi_2^\bot\|_{\rho-\delta}  \leq  C \kappa \delta ^{-\nu}  \|K_1-K_2\|_{\rho,X}^2 $.

The condition on the right-hand side of \eqref{approxUnique} to have zero average gives
$|\avg(\xi_2)  | \le  C\kappa \delta ^{-\nu} \|K_1-K_2\|_{\rho,X}^2$. Then 
$$
\|\xi_1 -\avg(\xi_1) \|_{\rho-2\delta} \le C\kappa^2 \delta ^{-2\nu} \|K_1-K_2\|_{\rho,X}^2
$$
but $\avg(\xi_1) $ is free. Then 

\begin{equation*}
\|\Delta^c-(\avg(\Delta^c)_1,0)^\top \|_{\rho-2\delta} \leq C \kappa
^2 \delta ^{-2\nu } \|K_1-K_2\|_{\rho,X}^2.
\end{equation*}
The next step  is done in 
the same way as in \cite{LGJV05}. We quote Lemma 14 of that reference
using our notation.   
It is basically an application of the standard implicit function theorem. 

\begin{lemma} \label{lem:tau}
There exists a constant $C$ such that if $C \|K_1-K_2\|_{\rho,X} \leq 1$
then there exists an initial phase $\tau_1 \in \left \{ \tau   \in
  \mathbb{R}^\ell \mid \; |\tau|< \|K_1-K_2\|_{\rho,X} \right \}$ such that 
 \begin{equation*}
\avg(T_2(\th) \Pi^c_{K_2(\th)}(K_1\circ T_{\tau_1}-K_2)(\th))=0. 
\end{equation*}
\end{lemma}
The proof is based on 
the implicit function theorem in $\mathbb{R}^\ell$. 

As a
consequence of Lemma \ref{lem:tau}, if $\tau_1$ is as in the statement, then $K \circ
T_{\tau_1}$ is a solution of \eqref{eq:embedding} such that for all $\delta \in (0,\rho/2)$ 
we have the estimate 
\begin{equation*}
\|W\|_{\rho-2\delta,X} <C \kappa^2
\delta^{-2\nu} \|\mathcal{R}\|^2_{\rho} \leq C \kappa^2 \delta^{-2\nu}
\|K_1-K_2\|^2_{\rho,X}.  
\end{equation*}
This leads to, on the center subspace 
\begin{equation*}
\|\Pi^c_{K_2(\th)}(K_1\circ T_{\tau_1}-K_2)\|_{\rho-2\delta,X} \leq C \kappa^2
\delta^{-2\nu} \|K_1-K_2\|^2_{\rho,X}. 
\end{equation*}
Furthermore, taking 
projections on the  hyperbolic subspace, we have that
$\Delta^h=\Pi^h_{K_2(\th)} (K_1-K_2)$ satisfies the estimate 
\begin{equation*}
\|\Delta^h\|_{\rho-2\delta,X} <C \|\mathcal{R}\|_{\rho,Y}.
\end{equation*}
All in all, we have proven the estimate for $K_1 \circ T_{\tau_1}
-K_2$ (up to a change in the original constants)
\begin{equation*}
\|K_1\circ T_{\tau_1}-K_2\|_{\rho-2\delta,X} \leq C \kappa^2
\delta^{-2\nu} \|K_1-K_2\|^2_{\rho,X}. 
\end{equation*}
We are now in position to carry out an argument based on iteration. 
We can take a sequence $\left \{ \tau_m \right
\}_{m \geq 1}$ such that $|\tau_1| \le \|K_1  -K_2 \|_{\rho,X}$ and 
\begin{equation*}
|\tau_m -\tau_{m-1}| \leq \|K_1 \circ T_{\tau_{m-1}} -K_2 \|_{\rho_{m-1},X} , \qquad m\ge 2,
\end{equation*}
and 
\begin{equation*}
\|K_1 \circ T_{\tau_{m}} -K_2 \|_{\rho_{m},X} \leq C \kappa^2
\delta_m^{-2\nu} \|K_1 \circ T_{\tau_{m-1}} -K_2 \|^2_{\rho_{m-1},X},
\end{equation*}
where $\delta_1=\rho/4$, $\delta_{m+1}=\delta_m/2$ for $m\ge 1$ and 
$\rho_0=\rho $, $\rho_m=\rho_0-\sum_{k=1}^m \delta_k$ for $m\ge 1$.
By an induction argument we end up with 
\begin{equation*}
\|K_1 \circ T_{\tau_{m}} -K_2 \|_{\rho_{m},X} \leq
(C\kappa^2\delta^{-2\nu}_1 2^{2\nu} \|K_1-K_2\|_{\rho_0,X})^{2^m}
2^{-2\nu m}. 
\end{equation*} 
Therefore, under the smallness assumptions on $\|K_1-K_2\|_{\rho_0,X}$,
the sequence $\left \{\tau_m \right \}_{m \geq 1}$ converges and one gets 
\begin{equation*}
\|K_1 \circ T_{\tau_\infty} -K_2\|_{\rho /2, X}=0.
\end{equation*} 
Since both $K_1 \circ T_{\tau_\infty} $ and $ K_2$ are analytic in $D_\rho$ and coincide in 
$D_{\rho/2}$ we obtain 
the result.

\section{Nash-Moser iteration}
\label{iter}

In this section, we show that, if the initial error of the approximate
invariance equation \eqref{eq:embedding-approx} is small enough the
Newton procedure can be iterated infinitely many times and converges to
a solution.  This is somewhat standard in KAM theory given the
estimates already obtained.

Let $K_0$ be an approximate solution of \eqref{eq:embedding} (i.e. a solution
of the linearized equation with error $E_{0}$). We define the following sequence of approximate solutions
\begin{eqnarray*}
K_m=K_{m-1}+\Delta K_{m-1},\qquad m\ge 1,\\
\end{eqnarray*}
where $\Delta K_{m-1}$ is a solution of 
\begin{equation*}
D_{K}\mathcal{F}_{\omega}(K_{m-1})\Delta K_{m-1}=-E_{m-1}
\end{equation*}
with $E_{m-1}(\th)=\mathcal{F}_{\omega}(K_{m-1})(\th)$. The next lemma provides that the solution at step $m$
improves the solution at step $m-1$ and the norm of the error at step
$m$ is bounded in a smaller complex domain by the square of the norm of the error at step $m-1$.  

\begin{pro}\label{improvement}
Assume $K_{m-1} \in ND(\rho_{m-1})$ is an approximate solution of equation \eqref{eq:embedding}  
and that the following holds 
\begin{equation*}
r_{m-1}=\|K_{m-1}-K_0\|_{\rho_{m-1},X} < r.
\end{equation*} 
If $E_{m-1}$ is small enough such that Proposition \ref{represVF} applies, i.e.
$$
C\kappa\delta_{m-1}^{-\nu-1} \|E_{m-1}\|_{\rho_{m-1},Y} < 1/2
$$
for some $0< \delta_{m-1} \le \rho_{m-1}/3$,
then there exists a function $\Delta K_{m-1} \in
\A_{\rho_{m-1}-3\delta_{m-1},X}$ for some $0< \delta_{m-1} <
\rho_{m-1} /3$  such that 
\begin{equation}\label{improvement1}
\|\Delta K_{m-1}\|_{\rho_{m-1}-2\delta_{m-1},X} \leq (C^1_{m-1}+C^2_{m-1} \kappa^{2} \delta_{m-1}^{-2\nu}) \|E_{m-1}\|_{\rho_{m-1},Y},
\end{equation}

\begin{equation}\label{improvement2}
\|D\Delta K_{m-1}\|_{\rho_{m-1}-3\delta_{m-1},X} \leq (C^1_{m-1}\delta_{m-1}^{-1}+C^2_{m-1} \kappa^{2} \delta_{m-1}^{-(2\nu+1)}) \|E_{m-1}\|_{\rho_{m-1},Y},
\end{equation}
where $C^1_{m-1}, C^2_{m-1}$ depend only on $\nu$, $l$,
$|\X|_{C^1(B_r)}$, $\|DK_{m-1}\|_{\rho_{m-1},X}$,
$\|\Pi^s_{K_{m-1}(\th)}\|_{\rho_{m-1},Y^s_\th,X}$,
$\|\Pi^c_{K_{m-1}(\th)}\|_{\rho_{m-1},Y^c_\th,X}$,
$\|\Pi^u_{K_{m-1}(\th)}\|_{\rho_{m-1},Y^u_\th,X}$,
and $|\avg(S_{m-1})|^{-1}$. Moreover, if $K_m=K_{m-1}+\Delta K_{m-1}$ and 
\begin{equation*}
r_{m-1}+(C^1_{m-1}+C^2_{m-1} \kappa^{2} \delta_{m-1}^{-2\nu}) \|E_{m-1}\|_{\rho_{m-1},Y} <r
\end{equation*} 
then we can redefine $C^1_{m-1}$ and $C^2_{m-1}$ and all previous
quantities such that the error
$E_m(\th)=\mathcal{F}_{\omega}(K_m)(\th)$ satisfies (defining $\rho_m=\rho_{m-1}-3\delta_{m-1}$) 
\begin{equation}\label{estimate}
\|E_m\|_{\rho_m,Y} \leq C_{m-1}\kappa^4 \delta_{m-1}^{-4\nu} \|E_{m-1}\|^2_{\rho_{m-1},Y}. 
\end{equation}
\end{pro}  
\begin{proof}  We have $\Delta K_{m-1}(\th)=
\Pi^h_{\th} \Delta K_{m-1}(\th) +\Pi^c_{\th} \Delta
K_{m-1}(\th)$, where $\Pi^h_{\th}$ is the projection on the
hyperbolic subspace and belong to $\L(Y^h_\th,X)$. Estimates \eqref{improve1}
follow from the previous two sections. The second part of 
estimate \eqref{improve1} follows from the first line of 
 \eqref{improve1}, Cauchy's inequalities and the fact that the projected equations on the hyperbolic subspace are exactly solved.  
\end{proof}  
Thanks to the previous proposition, one is able to obtain the convergence of the Newton method in a standard way.   

The others non-degeneracy conditions can be checked in exactly the
same way as described in \cite{fontichdelLS071} and we do not repeat the arguments. 
\begin{lemma}
If $\|E_{m-1}\|_{\rho_{m-1},Y^c_\th }$ is small enough, then
\begin{itemize}
\item If $ DK_{m-1}^\perp  DK_{m-1}$ is invertible with inverse $N_{m-1}$
  then $$ DK_{m}^\perp DK_{m}$$ is invertible with inverse $N_m$ and we have    
\begin{equation*}
\|N_m\|_{\rho_{m}} \leq \|N_{m-1}\|_{\rho_{m-1}}+C_{m-1} \kappa^2 \delta_{m-1}^{-(2\nu+1)} \|E_{m-1}\|_{\rho_{m-1},Y^c_\th}.
\end{equation*}
\item If $\avg(S_{m-1})$ is non-singular then also $\avg (S_{m})$ is and we have the estimate
 \begin{equation*}
|\avg(S_m)|^{-1}\leq |\avg(S_{m-1})|^{-1}+C'_{m-1} \kappa^2 \delta_{m-1}^{-(2\nu+1)} \|E_{m-1}\|_{\rho_{m-1},Y^c_\th}.
\end{equation*}
\end{itemize}
\end{lemma}

\section{Construction of quasi-periodic solutions for the Boussinesq equation}\label{applications}

This section is devoted to an application of Theorem~\ref{existence}
to a concrete equation that has appeared in the literature. 

In Section~\ref{sec:formal}, we 
 will verify the formal hypothesis of the general Theorem~\ref{existence}. 
First we will verify the geometric  hypothesis, choose the concrete 
spaces that will play the role of the abstract ones, etc.
In Section~\ref{sec:Lindstedt}, we will construct approximate solutions 
that satisfy the quantitative properties. By applying Theorem~\ref{existence},
to these approximate solutions, we will obtain Theorem~\ref{thBou}. 

\subsection{Formal and geometric considerations}\label{sec:formal}

The Boussinesq equation has been widely studied in the context of fluid
mechanics since the pioneering work~\cite{Boussinesq}. It is the 
equation (in one dimension) with periodic boundary conditions
\begin{equation}\label{boussinesq}
u_{tt}= \mu u_{xxxx}+u_{xx}+ (u^2)_{xx}\,\,\,\,\mbox{on
  $\mathbb{T},\,\,t\in \RR$}. 
\end{equation}
where $\mu >0$ is a parameter.

We will introduce an additional parameter $\ep$ which will
 be useful in the sequel as a nemonic 
device to perform perturbation theory.  
Note however that it can be eliminated by rescaling the $u$, considering $v=\ep u$. So 
that discussing small $\ep$ is equivalent to discussing small amplitude 
equations.

The  equation \eqref{boussinesq} is ill-posed in any space and one can construct initial data for which there is no existence 
in any finite interval of time. As we will see later, the non-linear 
term does not make it well posed in  the spaces $X$ we will consider 
later. 

The equation~\eqref{boussinesq} 
 is a $4$th order equation in space.  Since it is second order 
in time, it is standard to write it as a first order system
\begin{equation}\label{LN}
\dot{z}=\L_\mu z + \mathcal N(z),
\end{equation}
where 
$$\L_\mu=\begin{pmatrix} 0  & 1 \\ \partial^2_{x} +\mu \partial^4_x& 0 \end{pmatrix}$$
and 
$$\mathcal N(z)=(0,\partial^2_x u^2).$$

Notice that \eqref{LN} has the structure we assumed in 
\eqref{decomposition} , namely that the evolution operator is 
the sum of a linear and constant operator and a nonlinear part,
which is of lower order than the linear part.

\subsection{Choice of spaces}
In this section we present some choices of 
spaces $X$,$Y$ for which the operators entering in the Boussinesq
equation satisfy the assumptions of Theorem~\ref{existence}.
As indicated in Section~\ref{sec:consequences}, 
there are several choices
and it is advantageous to follow a choice for the local uniqueness part 
and a different one for the existence. The spaces we consider 
will have one free parameter.

For  $\rho >0$  we denote:
$$D_{\rho}= \left \{ z \in \CC^\ell/ \ZZ^\ell\,|\, |\mbox{Im}\, z_i| <\rho \right \}$$
and denote $H^{\rho,m}(\TT)$ for $\rho >0$ and $m \in \NN$, the space of analytic functions $f$ in $D_\rho$ such that the quantity 
$$\|f\|^2_{\rho,m}=\sum_{k \in \ZZ} |f_k|^2 e^{4\pi \rho |k|}(|k|^{2m}+1)$$
is finite, and where $\left \{ f_k \right \}_{k \in \ZZ}$ are the
Fourier coefficients of $f$. For any $\rho >0$ and $m \in \NN$, the
space $\Big (H^{\rho,m}(\TT), \| \cdot \|_{\rho,m} \Big )$ is a
Hilbert space. Furthermore, this scale of Hilbert spaces
$H^{\rho,m}(\TT)$ for $\rho >0$ and $m > \frac12$ is actually a
Hilbert algebra for pointwise multiplication, i.e. for every $u,v \in
H^{\rho,m}(\TT)$ there exists a constant $C$ such that
$$\|u \, v \|_{\rho, m} \leq C \|u \|_{\rho, m} \|v \|_{\rho, m}.$$  

Extending the definition to $\rho=0$, $H^{0,m}(\TT)$ is the standard
Sobolev space on the torus and for $\rho >0$,  $H^{\rho,m}(\TT)$ consists of
analytic functions on the extended strip $D_\rho$ with some
$L^2$-integrability conditions on the derivatives up to order $m$ on
the strip $D_\rho$.  As already noticed, we are going to construct
quasi-periodic solutions in the class of small amplitude solutions for
\eqref{boussinesq}.

For the system~\eqref{LN}, it is natural to consider the space for $\rho >0$ and $m>\frac52$
\begin{equation} \label{Xspace}
X_{\rho,m} = H^{\rho,m} \times H^{\rho, m-2}
\end{equation}

We note that ${\mathcal L}_\mu$ sends $X_{\rho, m}$ into $X_{\rho,
  m-2}$, but we observe that this is not really used in Theorem
~\ref{existence}. By the Banach algebra property of the scale of
spaces $H^{\rho,m}(\TT)$ when $m > 1/2$ and the particular form of the
nonlinearity, we have the following proposition (see \cite{llave08}).
\begin{pro}
The non linearity $\mathcal N$ is analytic from $X_{\rho, m}$ into $X_{\rho,m}$
when $m > 5/2$. 
\end{pro}

In the system language, 
it is useful to think of ${\mathcal L}_\mu$ as an operator of order $2$ and
of $\mathcal N$ as an operator of order $0$.

Hence, in the present case, we can take $Y = X$ in the abstract Theorem \ref{existence}. 

\begin{remark} \label{illposed-nonlinear} 
Note that this gives a rigorous proof that the nonlinar evolution is 
ill-posed. If the non-linear evolution was well-posed in some of the 
$X_{\rho, m}$ spaces with $m > 5/2$, we could consider the nonlinear
evolution as a perturbation of the linear one. Using the 
usual Duhamel formula  of Lipschitz perturbations of 
semigroups \cite{henry},  we could conclude that the linear evolution
is well posed, which is patently false. 
\end{remark}

We will be actually considering a subspace of $X$ denoted $X_0$ consisting of functions $z(t) \in X$ such that 
\begin{equation} \label{automaticsymetry2} 
\int_0^1 dx \, z(\cdot ,x)\,dx=0.
\end{equation}
\begin{equation} \label{eq:momentum}
\int_0^1 dx \,  \partial_t z(\cdot,x)=0.
\end{equation} 
\begin{equation} \label{eq:symmetry} 
z(\cdot , x) = z(\cdot, -x)
\end{equation}

At the formal level, the subspace $X_0$ is invariant under
the equation  of \eqref{LN}.  
In contrast with the normalizations \eqref{automaticsymetry2} and 
\eqref{eq:momentum} that can be enforced by a change of variables, 
\eqref{eq:symmetry} is a real restriction.  It is possible 
to develop a theory without \eqref{eq:symmetry}, but we will not 
pursue it here.

We now check that the assumptions of Theorem \ref{existence} are
met. The main steps are to verify the formal assumptions of Theorem \ref{existence} and construct approximate solutions which are non degenerate.  

\subsection{Linearization around $0$}
We first study the eigenvalue problem for $U \in X, \sigma \in \CC$

\begin{equation*}
\L_\mu U = \sigma U. 
\end{equation*} 
This leads to the eigenvalue relation 
$$\sigma^2=-4\pi^2 k^2+ 16\pi^4\mu k^4=-4\pi^2k^2(1-4\pi^2 \mu k^2)$$
for $k \in \mathbb{Z}$. By symmetry, we assume that $k \geq 0$ and the
spectrum follows by reflection with respect to the imaginary axis. We
have the following lemma. 
\begin{lemma}
The operator $\L_\mu$ has discrete spectrum in $X$. Furthermore,
we have the following 
\begin{itemize}
\item The center spectrum of $\L_\mu$ consists in a finite number
  of eigenvalues. Furthermore, the dimension of the center subspace is
  even. 
\item The hyperbolic spectrum is well separated from the center
  spectrum. 
\end{itemize} 
\end{lemma} 
\begin{proof}
From the equation, 
$$\sigma^2=-4\pi^2 k^2+16\pi^4 \mu k^4=-4\pi^2k^2(1-4\pi^2 \mu k^2)$$
we deduce easily that the spectrum is discrete in $X$. Furthermore,
$0$ is not an eigenvalue since we assume $u$ to have average $0$. Finally, we notice that when
$0<k^2<\frac{1}{4\pi^2 \mu}$, one has $\sigma^2 <0$ and since there is a
finite (even) number of values in this set, this leads to the desired
result.  The separation of the spectrum directly follows from the discreteness of the spectrum. 
\end{proof}
We then have the following set of eigenvalues 
$$\text{Spec}(\L_\mu)=\left \{ \pm 2\pi i |k| \sqrt{1-4\pi^2 \mu k^2} =\pm \sigma_k(\mu)\right \}_{k\geq 1}. $$

The center space $X^c_0$ is the eigenspace degenerate by the
eigenfunctions correponding to the eigenvalues $\sigma_k(\mu)$ for
which indices $k=1,...,\ell$ we have $1-4\pi^2 \mu k^2 \ge 0$. The
center subspace $X_0^c$ is spanned by the eigenvectors

\[
U_k=(u_k,v_k)=(\cos( 2\pi k x), \sigma_k(\mu) \cos(2\pi k x))_{k=1, \ldots,\ell}.
\]

Any element $U$ on the center subspace can be expressed as:
\[
U=\sum_{k=1}^{\ell} \alpha_k U_k.
\]
with the $\alpha_k$ arbitrary real numbers. 

\subsection{Verifying the smoothing properties of
the partial evolutions of the linearization around $0$} 

We now come to the evolution operators and their smoothing properties. 
We have:

\begin{lemma}
The operator $\L_\mu$ generates  semi-group operators $U_\th^{s,u}(t)$ in positive and negative times. Furthermore, the following estimates hold
$$\|U_\th^{s}(t)\|_{X,X} \leq \frac{C}{t^\frac12} e^{-D t },t >0$$
and 
 $$\|U_\th^{u}(t)\|_{X,X} \leq \frac{C'}{|t|^\frac12} e^{D' t },t <0$$
for some constants $C,C',D,D'>0$. 
\end{lemma}

\begin{proof} 
The proof is given in 
detail in  \cite[page 404-405]{llave08}. It is based on  observing that the evolution operator 
in the (un)stable spaces can be expressed in 
Fourier series. 
Since the norms considered are 
given by the Fourier terms (with different weights), it suffices 
to estimate the sup of the multipliers times the ratio of the weights. 
\end{proof}

Until now, we have considered only the
 linearization around the equilibrium $0$ in $X$. 
Of course, by the stability theory of the splittings developed in 
Section~\ref{change-nondeg}, the spectral non-degeneracy properties will 
be satisfied by all the approximate solutions that are small enough in 
the smooth norms. As we will see, our approximate solutions will 
be trigonometric polynomials with small coefficients.

\subsection{Construction of an approximate solution} 
\label{sec:Lindstedt}

This section is devoted to the construction of an approximate non-degenerate solution for equation \eqref{LN}. We use a Lindstedt series
argument to construct approximate solutions for 
all \emph{``nonresonant''} values of $\mu$. 
Then, we will verify the twist  non-degeneracy conditions 
for some values of $\mu$ only.

\begin{remark}
For the experts, we note that the analysis
is remarkably similar to the perturbative analysis near elliptic 
fixed points in Hamiltonian systems. We have found useful  the treatment in \cite[Vol 2]{Poincare99}. 
More modern treatments based on transformation theory are
in  \cite{Moser68, Zehnder73,Douady88b}. In our case, the transformation 
theory is more problematic,  hence we take advantage of the a-posteriori format and just construct approximate solutions for the initial guess. 
\end{remark}

The following result establishes the existence (and some uniqueness which we will not use) 
of the Lindstedt series
under appropriate non-resonance conditions.  
\begin{lemma}\label{lindApprox}
Let $\ell$ be as before. For all $N \ge 2$, 
assume the nonresonance condition to order $N$
given by 
\begin{equation*}
F(k,j) \ne 0, \quad \quad  k \in \ZZ^\ell, j \in \NN, 1 < |k| \le N
\end{equation*}
where 
$$
F(k,j)\equiv \left[ (\omega_0 \cdot k)^2 -  2 \pi^2( j^ 2 - 2 \mu\pi^2 j^2)\right]. 
$$

 Then, for all $\mathcal U_1$ depending on $\ell$ parameters, 
there exist $(\omega^1,...,\omega^N) \in (\RR^\ell)^N$ and $(\mathcal U_2,...,\mathcal U_N) \in (H^{\rho,m}(\TT))^{N-1}$ parametrized by $(A^1_1,...,A^1_\ell)\in \RR^\ell$ for any $\rho >0$ such that for any $\sigma \geq 0$
\[
\|(u_\ep^{[\le N]})_{tt} -(u_\ep^{[\le N]})_{xx}-\mu (u_\ep^{[\le N]})_{xxxx}-((u_\ep^{[\le N]})^2)_{xx}.\|_{H^{\rho,m}(\TT)} \leq C \ep^{N+1}
\]
for some constant $C>0$ and 
\[
u_\ep^{[\le N]}(t,x)=\sum_{k=1}^N\ep^k \mathcal U_k( \omega^{[\leq  N]}_\ep t,x)
\]
where 
\[
\omega^{[\leq  N]}_\ep=\omega^0+\sum_{k=1}^N \ep^k \omega^k. 
\]
The coefficients $\mathcal U_k$ are trigonometric polynomials and can be obtained in such a way 
that the projection over the kernel of 

$$\M_0 = (\omega^0 \cdot \partial_{\th})^2 -\partial^2_{xx}-\mu \partial^4_{xxxx} $$

 is zero.
Moreover, the normalizations \eqref{symmetrytime}, \eqref{automaticsymetry2}
are satisfied.   With such a normalization,
they are unique. 
\end{lemma}

Before going into the proof itself, we comment a bit on the theory of Lindstedt series. We define the hull function as 
$$u_\ep(t,x)=\mathcal U_\ep (\omega_\ep t, x)$$
where $\mathcal  U_\ep: \TT^\ell \times \TT\mapsto \mathbb R$ with $\ell=\frac{\text{dim} X_0^c}{2}$. 

There are two versions of the theory: one assuming the symmetry 
condition for the solutions 
\begin{equation} \label{symmetrytime} 
\mathcal  U_\ep(\th,\cdot)= \mathcal U_\ep(-\th,\cdot)
\end{equation}
and another one without assuming \eqref{symmetrytime}. 
For simplicity, we will assume the symmetry. We note that, thanks to 
the a-posteriori format of the theorem, we only need to produce 
an approximate solution and verify the non-degeneracy conditions.

The function $ \mathcal U_\ep$ and the frequency $\omega_\ep$ produce
a solution of \eqref{LN} if and only if they satisfy the equation 
\begin{equation}\label{eqU}
(\omega_\ep \cdot \partial_{\th})^2 \mathcal U_\ep =\partial^2_{xx} \mathcal U_\ep+ \mu \partial^4_{xxxx} \mathcal U_\ep+ (\mathcal U_\ep^2)_{xx}.
\end{equation}
We emphasize that we are considering now 
that  both $\mathcal U_\ep$ and $\omega_\ep$ are unknowns to be determined in 
\eqref{eqU}. As we will see, 
we will obtain $\mathcal U_\ep$ and $\omega_\ep$, depending on 
$\ell$ free arbitrary parameters.

Following the standard procedure of Lindstedt series, 
we will consider formal expansions   $\mathcal U_\ep$
and 
$\omega_\ep$ in  powers of $\ep$. 
We will impose that finite order truncations to order $N$ satisfy the equation 
\eqref{eqU} up to an error $C_N |\ep|^{N+1}$. Hence, the series are 
not meant to converge (in general they will not) but they indicate a sequence 
of approximate solutions that solve the equation to higher and higher order
in $\ep$. We will also verify the other non-degeneracy hypothesis 
of Theorem~\ref{existence}. 

We consider the formal sums
\begin{equation} \label{formalsums}
\begin{split} 
& \mathcal U_\ep(\th,x) \sim \sum_{k=1}^\infty \ep^k \mathcal U_k (\th,x) \\
&\omega_\ep \sim \omega^0 +\sum_{k=1}^\infty \ep^k \omega^k.
\end{split} 
\end{equation}

\begin{remark}
Notice that the sum for $\mathcal U_\ep$ starts with $\ep$ since we have in mind to consider small amplitude solutions of the equation. 
\end{remark}

The meaning of formal power solutions is 
that  we 
truncate these sums at order $N$ arbitrary,  $N \geq 1$ and consider
\begin{equation*}
\begin{split}
&u^{[\leq N]}_\ep(\th,x)=\sum_{k=1}^N \ep^k \mathcal U_k (\th,x) \\
& \omega^{[\leq N]}_\ep = \omega^0+\sum_{k=1}^N  \ep^k \omega^k.
\end{split}
\end{equation*}

As  it often happens in Lindstedt series theory, the first terms of 
the recursion are different from the others. In our case, the first 
step will allow us to choose solutions of the first step depending 
on $\ell$ parameters. Once these solutions are chosen, 
we can obtain all the other solutions in a unique way. 
We note that the computations are very algorithmic and subsequently  
can be programmed. The normalization in the last item of Lemma~\ref{lindApprox} is 
natural in Lindstedt series theory. If one changes the parameters, 
introducing new parameters
$A^1_i = B^1_i + \ep \hat A_i(B^1_1, \ldots, B^1_\ell; \ep)$,
one obtains a totally different series, which of course parametrizes 
the same set of solutions. In any case, we emphasize that for 
us the main issue is to construct an approximate solution.

\begin{proof}
We substitute  the sums for $\omega_\ep$ and $\mathcal U_\ep$  into \eqref{eqU} and identify at all orders. 

{\bf Order $1$}: We get $$(\omega_0 \cdot \partial_{\th})^2 \mathcal U_1=\partial^2_{xx} \mathcal U_1 +\mu \partial^4_{xxxx} \mathcal U_1. $$

\def\M{{\mathcal M}}
We search for solutions of the form $\cos(2\pi \omega^0_j \th_j) \cos(2\pi j x)$
where $j \in \NN$. Therefore the frequencies are given by the relation 
$$\omega^0_j = 2\pi |j| \sqrt{1-4\pi^2 \mu j^2} .$$

We assume now that $4\pi^2\mu j^2 \neq 1$ and $1-4\pi^2\mu j^2 \geq 0$
which means that $j=1,...,\ell$ where
$\ell=\lfloor{\sqrt{\frac{1}{2\pi \mu}}}\rfloor$.

Now, we get the frequency vector
$\omega^0$, given by: 

$$(\omega^0)_{j=1,...,\ell}=\Big (2\pi |j| \sqrt{1-4\pi^2 \mu j^2}\Big )_{j=1,...,\ell}.$$

All the solutions of the equation satisfying the symmetry conditions 
\eqref{automaticsymetry}, \eqref{symmetrytime} are given by: 

\begin{equation}\label{order1}
\mathcal U_1 (\th,x)=\sum_{j=1}^\ell A^1_j \cos (2\pi \th_j) \cos (2\pi j x). 
\end{equation}

This is the customary analysis of the linearized equations in normal 
modes. For future reference, we denote 
$$\M_0 = (\omega_0 \cdot \partial_{\th})^2 -\partial^2_{xx}-\mu \partial^4_{xxxx}. $$

We note that the operator $\M_0$ is diagonal on trigonometric 
polynomials and we have that 
\begin{eqnarray} \label{multiplier} 
\mathcal M_0   \cos( 2 \pi k \cdot \th) \cos( 2 \pi j x)  =  F(k,j) \cos( 2 \pi k \cdot \th) \cos( 2 \pi j x) \nonumber
\end{eqnarray}
where 
$$
F(k,j)\equiv \left[ (\omega_0 \cdot k)^2 -  2 \pi^2( j^ 2 - 2 \mu\pi^2 j^2)\right]. 
$$

For convenience we will make the important {\bf non-resonance condition} to order $N$
\begin{equation} \label{nonresonance} 
F(k,j) \ne 0, \quad \quad  k \in \ZZ^\ell, j \in \NN, 1 < |k| \le N. 
\end{equation}

The nonresonance condition is very customary in the  study of 
elliptic fixed points. It says that the basic frequencies are 
not a combination of each other.  Note that if we fix $\ell, k$ and $j$ 
the condition $F(k,j) = 0$ is a polynomial equation in $\mu$ so that it is 
satisfied only for a finite number of $\mu$. This says that for the 
interval of $\mu$ where $\ell$ is constant, we may have to exclude
at most a finite number of values of $\mu$. Of course, requiring the result for all $N$ means excluding at most a countable number of values of $\mu$. 
A detailed analysis may obtain sharper conclusions 
on the values of $\mu$ that need to be excluded. 
In the final applications, we will only consider the interval in 
which $\ell = 1$, where it is easy to see that there is no resonant value. The following remark is obvious, but it will be useful for 
us later: 
\begin{pro} \label{prop:kernel}
Under the non-resonance condition, the kernel of 
the operator $\M_0$  is precisely $\omega^0\cdot \partial_\th$ of the span of 
the solutions $\mathcal U_1$ obtained before in \eqref{order1}. 
\end{pro}

{\bf Order $m \geq 2$: } The general equation  to be solved at order 
$m$ to ensure that the equation \eqref{eqU} is solvable to order 
$m$ has the form 
\begin{equation}\label{generalorder}
\mathcal M_0 \mathcal U_m  + 2 (\omega^{m-1} \cdot \partial_\th ) (\omega^0 \cdot \partial_{\th } )\mathcal U_1 =\mathcal{R}_m (\mathcal U_1,...,\mathcal U_{m-1},\omega^0,...,\omega^{m-2})
\end{equation}

where $\mathcal R_m$ is polynomial in its arguments and their
derivatives (up to order $4$). In particular, 
 if $\mathcal U_1,...,\mathcal U_{m-1}$ are
trigonometric polynomials then so is $\mathcal R_m$. It is also easy to see that if 
 $\mathcal U_1,...,\mathcal U_{m-1}$ have the symmetry properties
\eqref{symmetrytime} so does $\mathcal R_m$.  Hence, using the 
addition formula for products of angles, we can express 
\[
\mathcal R_m = \sum_{k \in \ZZ^\ell,j\in \ZZ} C_{k,j}(A^0_1, \ldots, A^0_\ell) \cos( 2 \pi k \cdot \th) \cos( 2\pi jx) 
\]

We inductively assume that  $\mathcal U_1, \ldots, \mathcal U_{m-1}$
are trigonometric polynomials and that $\omega^0,\ldots, \omega^{m-2}$ 
have been found. 
Then, we will show that  we can find $\omega^{m-1}$, $\mathcal U_m$ 
in such a way that the equation \eqref{generalorder} is solvable. 
Furthermore, the solution is unique if we impose the normalization 
at the end of Lemma~\ref{lindApprox}. The equation \eqref{generalorder} can be solved by 
identifying the coefficients of $\cos(2 \pi k \cdot \th) \cos(2 \pi j x)$
on both sides.

Since $\mathcal R_m$ is a trigonometric polynomial, we can separate the 
terms into terms that are in the kernel of $\mathcal M_0$ and 
terms for which the multiplier $F(k,j)$ corresponding to $\mathcal M_0$
is not zero.  We also note that, under the non-resonance hypothesis, 
we have that the kernel of $\mathcal M_0$ is precisely the 
functions that appear in $\mathcal U_1$. The term $(\omega^{m-1}\cdot \partial_\th) (\omega^{0}\cdot \partial_\th)\mathcal U_1$
lies in the kernel of $\mathcal M_0$.

Since $\mathcal M_0$ is diagonal, 
the  terms in the kernel of $\mathcal M_0$ are precisely those 
that are not in the range of $\mathcal M_0$. For the terms for which the multiplier  $F(k,j)$ is non zero 
(i.e. those terms in the range of $\mathcal M_0$),  we can invert 
$\mathcal M_0$ and, hence obtaining 
\[
\mathcal U_m(k,j)  = \frac{C_{k,j}}{F(k,j)} . 
\]

For the terms that lie in the kernel of $\mathcal M_0$, we cannot
divide by the multiplier $F(k,j)$ but instead obtain uniquely  $\omega^{m-1}$ to solve 
\eqref{generalorder}. Note that this uses the non-resonance 
condition so that  that the
kernel of $\mathcal M_0$ is precisely functions that appear
in $\mathcal U_1$. 

Of course, to solve \eqref{generalorder}, we could add any function in the 
kernel of $\mathcal M_0$. Under the normalization condition, we see 
that the term to add is uniquely determined to be zero. The evaluation of the norm in the Lemma comes directly from the fact that we are dealing with trigonometric polynomials, hence belonging to any Sobolev space.  
\end{proof}

\subsection{Application of Theorem  \ref{existence} to
the approximate solutions. End of 
the proof of  Theorem \ref{thBou}}
Let $\omega^0$ as in Theorem \ref{thBou} and consider $\mathcal U_\ep$  the function constructed in the previous section. Denote
$$K_0(\th)=\begin{pmatrix} \mathcal U_\ep (\th,.) \\ \omega_\ep \cdot \partial_\th \mathcal U_\ep (\th,.)\end{pmatrix} \in X_0. $$. 

We will proceed to verify the assumptions of 
Theorem~\ref{existence} taking as initial conditions of 
the iteration the results of the Lindstedt series. This will require carrying out explicitly the calculations indicated before to order $3$ and verifying that the twist condition is satisfied. 

\subsubsection{Smallness assumption on the error and range of $K_0$}

 Consider $K_0$ as above. Then Lemma \ref{lindApprox} ensures directly that the smallness assumption in Theorem \ref{existence} are satisfied with an error smaller than $C_N|\ep|^{N+1}$ for arbitrary large $N$.  

Note that this is verified for all values of $\ell$.

\subsubsection{Spectral non-degeneracy}

We check conditions \ref{ND1}. For $\ep=0$, all the conditions in \ref{ND1} are met by the previous discussion. In particular there exists an invariant splitting denoted 
\begin{equation}\label{split0}
X_0=X^c_0 \oplus X^s_0 \oplus X^u_0. 
\end{equation}

Now, by construction of $K_0$, choosing $\ep $ small enough again and
using the perturbation theory of the bundles developped in section
\ref{change-nondeg} (see Lemma~\ref{cor1:iterNH}), there exists an
invariant splitting for $K_0$ for $\ep$ small enough satisfying all
the desired properties and this proves the spectral non-degeneracy
conditions \ref{ND1} for $K_0$, together with the suitable estimates.

Note that this is verified for all values of $\ell$. 
\subsubsection{Twist condition}
 We now check the twist condition in Definition \ref{ND2}. Pick a Diophantine frequency $\omega $ as in Theorem \ref{thBou}.  Recall that the family of perturbative solutions is parameterized by 
$A^1_j$ for $j=1,...,\ell$, the $\ell$ parameters giving $\mathcal U_1$. In the system of coordinates given by $(A^1_1,\ldots, \theta)$, the twist condition amounts to showing that
\begin{equation}\label{twistinalpha}
|\mbox{det}\Big ( \partial_{A_j^1} \omega_i^N\Big )|^{-1} > T_N(\ep) > 0. 
\end{equation}

To verify the twist condition, we will assume that $\ell =1$. This is the only reason why in Theorem \ref{thBou} we are assuming $\ell=1$. 

If we can show that $T_N(\ep) >  C |\ep|^a$ for some positive 
$a, C $, $(1 \le a < N)$ then we claim that we can finish the construction. 
The crucial remark is that we also have 
\[
T_{\tilde N}(\ep)  \ge \tilde C |\ep|^a 
\]
for any $\tilde N > N$ since we are only adding higher order terms. As we will see $\omega^1=0$ so we will have to go to order $3$. Let us first consider the case $m=2$. We have that the equation at order $2$ and assuming that $\ell=1$ writes
$$
\mathcal M_0 \mathcal U_2 + 2 (\omega^{1} \cdot \partial_\th ) (\omega^0 \cdot \partial_{\th } )\mathcal U_1=(\mathcal U_1^2)_{xx}. 
$$
We have 
$$
\mathcal U^2_1= A^2 \cos^2(2\pi \th) \cos^2(2\pi x). 
$$
It yields 
$$
\mathcal U^2_1=\frac{A^2}{4}(1+\cos(4\pi \th))(1+\cos(4\pi x))
$$
and
 $$
 (\mathcal U_1^2)_{xx}=-4\pi^2 A^2 (1+\cos(4\pi \th))\cos(4\pi x), 
 $$
 since this is not in the range, hence one has $\omega^1=0$. We then go to order $m=3$ which gives  the equation (taking into account that $\omega^1=0$)
 $$
 \mathcal M_0 \mathcal U_3+2(\omega^0 \cdot \partial_\th)(\omega^2 \cdot \partial_\th)\mathcal U_1=2(\mathcal U_1\mathcal U_2)_{xx}. 
 $$
 From the previous step, one has 
 $$
 \mathcal U_2=-4\pi^2 A^2\Big ( \frac{\cos(4\pi x)\cos(4 \pi \th)}{F(2,2)}+\frac{\cos(4\pi x)}{F(0,2)} \Big )
 $$
 Hence we have 
 $$
 (\mathcal U_1 \mathcal U_2)_{xx}=-4 \pi^2 A^4 \Big ( -\cos(2\pi x)-9\cos(6\pi x)\Big )\Big ( \frac{-(\cos(2\pi \th)+\cos(6\pi \th))}{4F(2,2)}+\frac{\cos(2 \pi \th)}{2F(0,2)}\Big ).
 $$
 Identifying according to the discussion before, one gets that $\omega^2$ is given by 
 $$
 \omega^2=CA^4\Big ( \frac14 \frac{1}{F(2,2)}-\frac12 \frac{1}{F(0,2)}\Big ). 
 $$
 for some constant $C$. We check now that $\Big ( \frac14 \frac{1}{F(2,2)}-\frac12 \frac{1}{F(0,2)}\Big ) \neq 0$. We compute 
 $$
 F(0,2)-2F(2,2)=-12-8\mu \neq 0,
 $$
 hence $\omega^2 \neq 0$. 
 
 As a consequence one has 
 $$
 \omega^{[\leq N]}_\ep=\omega^0+\ep^2\omega^2+h.o.t. 
 $$
 and furthermore $\omega^2 \neq 0$. Since $\omega^0$ does not depend on $A$, we have that the twist condition writes
 $$
 \ep^2 \Big ( \frac{ d\omega^2}{dA} \Big )+h.o.t.
 $$
Hence, taking $\tilde N$ sufficiently large, we can apply Theorem~\ref{main} to obtain Theorem \ref{thBou}.

\section{Application to the Boussinesq system}
\label{applications2}

In this section, we consider the Boussinesq system of water
waves. This system is even more interesting than the Boussinesq
equation (see Section \ref{applications}) for at least two reasons:
first the system is more "singular"; second, the full power of the
{\sl two spaces} approach has to be used, i.e. one has to take the
spaces $X$ and $Y$ such that $X\neq Y$. The system writes
\begin{equation}\label{systemBouTemp}
\partial_t    
\begin{pmatrix} 
u  \\
v \\
\end{pmatrix}
=
\begin{pmatrix} 
0 & -\partial_x -\mu \partial_{xxx}  \\
-\partial_x & 0 \\
\end{pmatrix}
\begin{pmatrix} 
u \\
v \\
\end{pmatrix}
+
\begin{pmatrix} 
\partial_x (uv) \\
0 \\
\end{pmatrix}
\end{equation}
where $t>0$ and $x \in \TT.$ 

The elementary
 linear analysis around the $(0,0)$ equilibrium can be found in
\cite{llave08}.
Recall that the eigenvalues  of the linearization around $0$ are given by
\begin{equation}  
\omega(k)  = \pm |k| 2 \pi i
\sqrt{ 1 - 4 \pi^2 \mu k^2} \quad 
k \in \integer 
\end{equation} 

 The eigenvectors are given by 
$$
U_j=(2\pi j \cos(2\pi \th_j)\cos(2\pi j),\sqrt{(2\pi j)^2-\mu (2\pi j x )^4}\sin(2\pi \th_j)\sin(2\pi jx))
$$
for $j=1,...\ell$ where $\ell$ is the smallest integer such that $1-4\pi^2\mu k \geq 0$.  
 
We denote by $\omega^0$ the vector whose components are
all the real frequencies that appear
\begin{equation}\label{omega0} 
\begin{split}
\omega^0 = (\omega(k_1), \omega(k_2), \ldots, \omega(k_\ell));
&\\
\{k_1,\ldots k_\ell\} = \{k \in \ZZ \, | \, k > 0 ; 
  1 - 4 \pi^2 \mu k^2 \ge 0 \}
\end{split} 
\end{equation}

The following symmetries are preserved formally by the equation 
\begin{equation}\label{symWW}
\left \{ 
\begin{array}{c}
u(t,-x)=u(-t,x)=u(t,x),\\
v(t,-x)=v(-t,x)=-v(t,x). 
\end{array} \right . 
\end{equation}
We remind that we take 
$$
X=H^{\rho,m}(\TT) \times H^{\rho,m+1}(\TT)
$$
and 
$$
Y=H^{\rho,m-1}(\TT) \times H^{\rho,m}(\TT)
$$

We denote by $X_0$ the set of functions in $X$ satisfying the symmetries \eqref{symWW} and also the momentum
$$
\int_0^1 u(t,x)\,dx=0
$$
and 
$$
\int_0^1 v(t,x)\,dx=0. 
$$
The previous quantities, as in the case of the Boussinesq equation, are preserved by the equation under consideration. It is proved in \cite{llave08} the following proposition
\begin{pro}
The nonlinearity $\mathcal N(u,v)=(\partial_x (uv),0)$ is analytic (indeed a polynomial)  from $X$ to $Y$. 
\end{pro}

Furthermore one has (see also \cite{llave08}) 
 \begin{lemma}
 For  $t>0$, one has 
$$
\|U_\th^s(t)\|_{Y,X}\leq \frac{C}{t^{1/2}}e^{-D t}
$$
and for $t<0$ one has
$$
\|U_\th^u(t)\|_{Y,X}\leq \frac{C'}{|t|^{1/2}}e^{D' t}
$$
for some $C,C',D,D'>0$. 
\end{lemma}

\subsubsection{Approximate solution}
We will not repeat the whole discussion which is very close to the one on the Boussinesq equation. Instead, we provide the necessary changes. The strategy is completely parallel to the one for the Boussinesq equation. Define two hull functions
$$
u_\ep(t,x)=\mathcal U_\ep(\omega_\ep t,x)
$$
and 
$$
v_\ep(t,x)=\mathcal V_\ep(\omega_\ep t,x)
$$
Once again we consider Lindstedt series in powers of $\ep$.

 Similarly to the previous section, we have 
\begin{lemma}\label{lindApprox2}
Let $\ell$ be as before. 
For all $N >  1$, there exists $(\omega^1,...,\omega^N) \in (\RR^\ell)^N$, $(\mathcal U_1,...,\mathcal U_N) \in (H^{\rho,m}(\TT))^N$ and $(\mathcal V_1,...,\mathcal V_N) \in (H^{\rho,m-1}(\TT))^N$for some $\rho >0$ such that 
\begin{equation}\label{systemBouApprox}
\Big \| \partial_t    
\begin{pmatrix} 
u_\ep  \\
v_\ep \\
\end{pmatrix}
-
\begin{pmatrix} 
0 & -\partial_x -\mu \partial_{xxx}  \\
-\partial_x & 0 \\
\end{pmatrix}
\begin{pmatrix} 
u_\ep \\
v_\ep \\
\end{pmatrix}
+
\begin{pmatrix} 
\partial_x (u_\ep v_\ep) \\
0 \\
\end{pmatrix}
\Big \|_{H^{\rho,m}(\TT)\times H^{\rho,m-1}(\TT)} \leq C \ep^{N+1}
\end{equation}
for some constant $C>0$ and 
$$u^{[\leq N]}_\ep(t,x)=\sum_{k=1}^N\ep^k \mathcal U_k( \omega^{[\leq N]}_\ep t,x),$$
$$v^{[\leq N]}_\ep(t,x)=\sum_{k=1}^N\ep^k \mathcal V_k( \omega^{[\leq N]}_\ep t,x),$$
where 
$$
\omega^{[\leq N]}_\ep=\omega^0+\sum_{k=1}^N \ep^k \omega^k. 
$$

The solutions depend on $\ell$ arbitrary parameters, where $\ell$ is the
number of the degrees of freedom of the kernel. 
\end{lemma}

\begin{proof}

We develop a general theory, parallel with the one of the Boussinesq equation in the previous section. The main new difficulties is that we are dealing with systems of equations and that the linear operator is not diagonal in an obvious sense.  Denote
$$
\mathcal A= 
\begin{pmatrix} 0 & -\partial_x -\mu \partial^3_x   \\
-\partial_x  & 0 \\
\end{pmatrix}
$$
At general order $m \geq 2$, we search for solutions of the form
$$
\mathcal U_m(\th,x)=\sum_{j \ZZ, k \ZZ^\ell} U^m_{k,j} \cos(2\pi k \cdot \th) \cos(2\pi jx)
$$
and
$$
\mathcal V_m(\th,x)=\sum_{j \ZZ, k \ZZ^\ell} V^m_{k,j} \sin(2\pi k \cdot \th) \sin(2\pi jx).
$$
The previous formulae come from the assumptions of symmetry of the solutions. Denoting $\mathcal W_m=(\mathcal U_m,\mathcal V_m)$ one has
$$
\Big (\omega^0\cdot \partial_\th -\mathcal A \Big ) 
\mathcal W_m
+\omega^{m-1}\cdot \partial_\th 
\mathcal W_1=\mathcal R_m (\omega^0,..., \omega^{m-2}, \mathcal W_{m-1}). 
$$

It is important to notice the operator $\M_0=\Big (\omega^0\cdot \partial_\th -\mathcal A \Big )$ is not self-adjoint in $X$ and does not act as a multiplication in an easy basis of vectors.  We then need to understand the range of this operator. Its domain is spanned by
 $$\Big ((\cos(2\pi k \cdot \th)\cos(2\pi j x),\sin(2\pi k \cdot \th)\sin(2\pi j x)) \Big ).$$ 
 The range is then the space of vector functions of the form of linear combinations of the basis 
$$\Big ( (\sin(2\pi k \cdot \th)\cos(2\pi j x),\cos(2\pi k \cdot \th)\sin(2\pi j x))\Big ).$$

{\bf Order 1} One has 
\begin{equation}
 \omega^0 \cdot \partial_\th    
\begin{pmatrix} 
\mathcal U_1  \\
\mathcal V_1 \\
\end{pmatrix}
=
\begin{pmatrix} 
-\partial_x \mathcal V_1-\mu \partial^3_x \mathcal V_1 \\
-\partial_x \mathcal U_1 \\
\end{pmatrix}
\end{equation}
We expand
$$
\mathcal U_1=\sum_{j=1}^\ell A^1_j \cos(2\pi \th_j)\cos( 2\pi jx)
$$
$$
\mathcal V_1=\sum_{j=1}^\ell B^1_j \sin(2\pi \th_j)\sin(2\pi jx)
$$
As in the case of the Boussinesq equation, this gives directly the vector $\omega^0$ and one can take any $A^1_j,B^1_j$. For convenience later, we assume  
$$
A^1_j \neq 0, B^1_j \neq 0,\,\,\,j=1,\ldots,\ell 
$$
The rest of the orders is like in the previous section on the Boussinesq equation. 
 \end{proof}

We now prove Theorem \ref{thBouSystem}, i.e. considering the case $\ell=1$. It amounts to apply the abstract theorem \ref{existence}. As in Section \ref{applications}, this is done by checking the twist condition, the rest of the proof being completely parallel. We have first 
$$
\mathcal W_1=A
\begin{pmatrix}
\cos(2\pi \theta) \cos(2\pi x) 2\pi \\
\sin(2\pi \theta) \sin(2\pi x) 2\pi \omega^0. \\
\end{pmatrix}
$$
For simplicity of writing we suppress the harmless parameter $A$. 

At order $2$, one has 
\begin{equation}
 \omega^0 \cdot \partial_\th    
\begin{pmatrix} 
\mathcal U_2 \\
\mathcal V_2 \\
\end{pmatrix}
+\omega^1 \cdot \partial_\th    
\begin{pmatrix} 
\mathcal U_1 \\
\mathcal V_1 \\
\end{pmatrix}
=
\begin{pmatrix} 
-\partial_x \mathcal V_2-\mu \partial^3_x \mathcal V_2+\partial_x(\mathcal U_1 \mathcal V_1) \\
-\partial_x \mathcal U_2 \\
\end{pmatrix}
\end{equation}

Furthermore, one has (the map $F(j,k)$ is defined as in the previous section)
$$
\partial_x(\mathcal U_1\mathcal V_1)=\frac12 \sin(4\pi \th)\sin(4\pi x). 
$$

This is never in the range of $\mathcal M_0 =\omega^0\cdot \partial_\th -\mathcal A$. Therefore, we obtain $\omega^1=0$. Additionally, one has 
$$
\mathcal  W_2=\frac{1}{F(2,2)}\begin{pmatrix} 
\frac12 \cos(4\pi \th)\cos(4\pi x)\\
\frac12\sin(4\pi \th)\sin(4\pi x)\omega^0
\end{pmatrix} +
\frac{1}{F(-2,2)}\begin{pmatrix} 
\frac12 \cos(4\pi \th)\cos(4\pi x)\\
-\frac12\sin(4\pi \th)\sin(4\pi x)\omega^0
\end{pmatrix}
$$
We go now to order $3$. We have 
\begin{equation}
 \mathcal M_0   
\begin{pmatrix} 
\mathcal U_3 \\
\mathcal V_3 \\
\end{pmatrix}
+\omega^2 \cdot \partial_\th    
\begin{pmatrix} 
\mathcal U_1 \\
\mathcal V_1 \\
\end{pmatrix}
=
\begin{pmatrix}
\partial_x(\mathcal U_1\mathcal V_2)+\partial_x(\mathcal U_2 \mathcal  V_1)\\
0
\end{pmatrix}
\end{equation}

We have by lengthy but straightforward computations
$$
\mathcal U_1\mathcal V_2=\frac18 \frac{1}{F(-2,2)}\Big (\sin (6\pi \th)-\sin(2\pi \th) \Big )\Big (\sin (6\pi x)-\sin(2\pi x) \Big )-
$$
$$
\frac{\omega^0}{8F(2,2)}\Big (\sin (6\pi \th)-\sin(2\pi \th) \Big )\Big (\sin (6\pi x)-\sin(2\pi x) \Big )
$$

Similarly
$$
\mathcal U_2 \mathcal V_1=\frac18 \frac{1}{F(2,2)}\Big (\sin (6\pi \th)-\sin(2\pi \th) \Big )\Big (\sin (6\pi x)-\sin(2\pi x) \Big )+
$$
$$
\frac{\omega^0}{8F(-2,2)}\Big (\sin (6\pi \th)-\sin(2\pi \th) \Big )\Big (\sin (6\pi x)-\sin(2\pi x) \Big )
$$

Hence one has 

\begin{equation} 
\begin{split}
\partial_x(\mathcal U_1&\mathcal V_2)+\partial_x(\mathcal U_2\mathcal V_1)\\
&=
frac18 \Big ( \frac{1}{F(-2,2)}-\frac{\omega^2}{F(2,2)}+  \frac{1}{F(-2,2)}-\frac{1}{F(2,2)} \Big )\Big (2\pi \sin(2\pi \th)\cos(2\pi x)\Big )+ R(\th,x)
\end{split} 
\end{equation} 
where $R(\th,x)$ is a trigonometric polynomial involving higher order frequencies. Since the coefficient 
$$
\frac{\pi}{4} \Big ( \frac{1}{F(-2,2)}-\frac{\omega^0}{F(2,2)}+ \frac{1}{F(-2,2)}-\frac{1}{F(2,2)} \Big )
$$
is non-zero only on a finite number of values of $\mu$, one deduces that $\omega^2$ is nonzero, hence the twist condition. The rest of the proof follows.

\bibliographystyle{alpha}
\bibliography{biblioPDE}

\end{document}